\documentclass[11pt,reqno]{amsart}
\usepackage{amssymb}
\usepackage{mathptmx}
\usepackage{mathtools}
\usepackage{mathrsfs}
\usepackage[all]{xy}
\usepackage{stmaryrd}
\usepackage{fancyhdr}
\usepackage{hyperref}
\usepackage{mathrsfs}
\usepackage{tikz-cd}
\usepackage{cite}
\usepackage{amsmath,amsfonts}
\usepackage{graphicx}
\usepackage{wrapfig}
\usepackage{array}
\usetikzlibrary{topaths,calc}
\usepackage{amsthm}
\usepackage[left=1.2in, right=1.2in, top=1in, bottom=1in]{geometry}
\usepackage{etoolbox}
\usepackage{hyperref}
\patchcmd{\section}{\scshape}{\bfseries}{}{}
\makeatletter
\renewcommand{\@secnumfont}{\bfseries}
\makeatother
\xyoption{all}
\DeclareMathOperator{\Aut}{Aut}

\theoremstyle{plain}

\input xy
\xyoption{all}

\theoremstyle{definition}
\newtheorem{mydef}{\textbf{Definition}}[section]
\newtheorem{myeg}[mydef]{\textbf{Example}}

\newtheorem{construction}[mydef]{\textbf{Construction}}

\newtheorem{rmk}[mydef]{\textbf{Remark}}

\theoremstyle{plain}
\newtheorem{mythm}[mydef]{\textbf{Theorem}}

\newtheorem*{nothma}{\textbf{Theorem A}}
\newtheorem*{nothmb}{\textbf{Theorem B}}

\newtheorem{lem}[mydef]{\textbf{Lemma}}
\newtheorem{pro}[mydef]{\textbf{Proposition}}

\newtheorem{cor}[mydef]{\textbf{Corollary}}

\begin{document}

\title{Artin-Ihara L-functions for hypergraphs}

\author{Mason Eyler}
\address{New Paltz High School, New Paltz, NY 12561, USA}
\email{masoneyler314@gmail.com}

\author{Jaiung Jun}
\address{Department of Mathematics, SUNY New Paltz, NY 12561, USA}
\email{junj@newpaltz.edu}

	\makeatletter
	\@namedef{subjclassname@2020}{%
		\textup{2020} Mathematics Subject Classification}
	\makeatother


\begin{abstract}
We generalize Artin-Ihara L-functions for graphs to hypergraphs by exploring several analogous notions, such as (unramified) Galois coverings and Frobenius elements. To a hypergraph $H$, one can naturally associate a bipartite graph $B_H$ encoding incidence relations of $H$. We study Artin-Ihara $L$-functions of hypergraphs $H$ by using Artin-Ihara $L$-functions of associated bipartite graphs $B_H$. As a result, we prove various properties for Artin-Ihara L-functions for hypergraphs. For instance, we prove that the Ihara zeta function of a hypergraph $H$ can be written as a product of Artin-Ihara $L$-functions. 
\end{abstract}

\maketitle


\section{Introduction}

Zeta functions are defined for various mathematical structures such as algebraic varieties, number fields, groups, matroids, and graphs. One intriguing aspect of zeta functions is that they encode various information about mathematical structures of interest. For example, the Euler characteristic of both an algebraic variety and a graph can be readily extracted from their corresponding zeta functions.

In several instances, zeta functions are defined by using a (suitably defined) notion of \emph{primes} such as prime numbers, prime ideals, and prime cycles (in graphs). Moreover, zeta functions often enjoy three fundamental properties: (1) \emph{Rationality}, (2) \emph{Functional equation}, and (3) Analogue of \emph{Riemann hypothesis}. For instance, the zeta function associated to a non-singular projective algebraic variety over a finite field satisfies the above three fundamental properties.

The \emph{Ihara zeta function} $\zeta_X(u)$ of a graph $X$ is defined by using prime cycles in $X$. $\zeta_X(u)$ satisfies rationality for any graph $X$, and it satisfies certain functional equations when $X$ is a regular graph. Furthermore, one can characterize the class of graphs satisfying an analogue of Riemann hypothesis, namely \emph{Ramanujan graphs}. These are examples of expander graphs, and in fact, an interesting higher dimensional analogue has been studied by A.~Kamber in \cite{kamber2016lp}.

Y.~Ihara \cite{ihara1966discrete} first associated a zeta function to a discrete torsion-free subgroup of $\textrm{PGL}_2(K)$, where $K$ is a non-Archimedean field with residue field $\mathbb{F}_q$, which can be seen as an analogue of the Selberg zeta function or the Riemann zeta function. J.~P.~Serre \cite{serre2002trees} pointed out that the Ihara zeta function is a zeta function of some regular graph. T.~Sunada \cite{sunada2006fundamental,sunada2006functions} gave the definition of the Ihara zeta function by using terminologies of graph theory and the graph theoretic proof of Ihara Theorem. K.~Hashimoto \cite{hashimoto1989zeta} gave the determinant expression of Hashimoto type for the Ihara zeta function of a general graph by using the edge matrix.

H.~Bass \cite{bass1992ihara} considered a more general case of a group $G$ acting on a locally finite tree $X$, building on the work of Hashimoto \cite{hashimoto1989zeta}, to associate a zeta function to the quotient $X/G$. Then, Bass proved an interesting result (among others) that the zeta function of $X/G$ is the reciprocal of an explicit polynomial (Theorem \ref{theorem: three-term determinant}). Later, in \cite{stark1996zeta}, H.~M.~Stark and A.~Terras provided elementary proofs for various results including Bass' determinant formula. 

One interesting facet of the story is that there are numerous graph-theoretic counterparts of theorems and notions found in number theory. For example, there is a graph theory version of Prime number theorem (\cite[Theorem 10.1]{terras2007zeta}), which can be proved by mimicking the proof of Prime number theorem for function fields \cite[Theorem 5.12]{rosen2002number} by using Ihara zeta functions in place of Hasse-Weil zeta functions.

In the framework of Stark and Terras, the Ihara zeta function of a graph may be viewed as an analogue of the Dedekind zeta function of a number field (via the Euler product). In fact, in \cite{stark2000zeta}, Stark and Terras further introduced notions of \emph{Galois coverings} and \emph{Artin-Ihara L-functions} for graphs and proved properties analogous to Artin L-functions, such as the Induction property (Theorem \ref{theorem: L function theorems for graphs}). In particular, for a Galois covering $\pi:Y \to X$ of graphs with the Galois group $G$, the Ihara zeta function $\zeta_X(u)$ (resp.~$\zeta_Y(u))$ is obtained by evaluating the Artin-Ihara L-function of the covering $\pi$ at the trivial representation (resp.~the right regular representation), proving that $\zeta_X(u)$ divides $\zeta_Y(u)$. Note that in \cite{zakharov2021zeta}, D.~Zakharov proved a similar divisibility result in a more general setting. 

A \emph{hypergraph} is a natural generalization of a graph in which an edge can have \emph{any} number of vertices. For a hypergraph $H$, one can construct a \emph{bipartite graph} $B_H$ whose vertices are the vertices and edges of $H$. The edges of $B_H$ encode the incidence relations of $H$. A natural question to ask is whether or not the theory of Ihara zeta functions for graphs can be generalized to the case of hypergraphs. 

In \cite{storm2006zeta}, C.~Storm generalized the notion of Ihara zeta functions from graphs to hypergraphs by naturally extending several key notions, thereby introducing a framework for defining Ihara zeta functions in the hypergraph context. Among many interesting results in \cite{storm2006zeta}, Storm showed that the Ihara zeta function $\zeta_H(u)$ of a hypergraph $H$ is precisely the function $\zeta_{B_H}(\sqrt{u})$, where $\zeta_{B_H}(u)$ is the Ihara zeta function of the bipartite graph $B_H$ associated to $H$. Storm also provided an example showing that there are Ihara zeta functions of hypergraphs which are not the Ihara zeta function of any graph. Moreover, Storm proved that for a $(d,r)$-regular hypergraph $X$, a ``modified Riemann hypothesis'' for the Ihara zeta function $\zeta_X(u)$ is true if and only if the hypergraph $X$ is Ramanujan in the sense of W.~Li and P.~Sol\'e \cite{sole1996spectra}.

In this paper, our goal is to introduce a notion of Artin-Ihara L-functions for hypergraphs and study their basic properties. To achieve this goal, we generalize several definitions, such as free Galois coverings (Definition \ref{definition: Galois group}) and Frobenius elements (Definition \ref{definiion: Frobenius}), to hypergraphs. We do this by exploring relations between hypergraphs and their associated bipartite graphs. For instance, we prove the association from $H$ to $B_H$ gives rise to a functor $\mathbf{B}$ from the category of hypergraphs to the category of graphs, which is faithful (but not full). Also, for a group $G$, a $G$-action on a hypergraph $H$ naturally induces a $G$-action on the associated bipartite graph $\mathbf{B}(H)=B_H$. 

Note that theories of hypergraph covering by using not the quotient hypergraph but the associated bipartite graph have been studied by various authors. For instance, I.~Sato \cite{sato2012edge}, D.~Li and Y.~Hou \cite{li2018hypergraph}, and D.~Li, Y.~Hou, Y.~Liao \cite{li2022zeta}.

By proving various relations between hypergraphs and associated bipartite graphs, we define sheet numbers for a free Galois covering of hypergraphs (Construction \ref{construction: sheet partition}). Consequently, we define the Artin-Ihara $L$-function of a free Galois covering $\pi:Y \to X$ of hypergraphs with the Galois group $G$ as follows (Definition \ref{definition: L-function}): for a representation $\rho$ of $G$, 
\[
L(u,\rho,Y/X):=\prod_{[C]} \det (1-\rho(F(C,Y/X))u^{\ell(C)})^{-1},
\]
where the product runs for all equivalence classes of prime cycles $C$ of $X$ (Definition \ref{definition: equivalence class of cycles}) and for each equivalence class $[C]$ we pick an arbitrary representative $C$. By appealing to the interplay between hypergraphs and associated bipartite graphs, we prove the following.

\begin{nothma}[Corollary \ref{corollary: same galois group} and Theorem \ref{theorem: linking L-functions}]
Let $Y$ be a connected hypergraph and $\pi:Y\to X$ be a free Galois covering of hypergraphs. Let $G$ be the Galois group of $\pi$. Then, one has the following:
\begin{enumerate}
    \item 
The induced map $\mathbf{B}(\pi):B_Y \to B_X$ is also a free Galois covering with the Galois group $G$. 
\item 
Let $\rho$ be a representation of $G$. Then, one has
\[
L(u,\rho,Y/X) = L(\sqrt{u},\rho,B_Y/B_X).
\] 
\end{enumerate}  
\end{nothma}

From the above theorem and the corresponding properties of Artin-Ihara $L$-functions of graphs, we obtain the following.

\begin{nothmb}[Corollaries \ref{corollary: 1} and \ref{corollary: 2}]
Let $Y$ be a connected hypergraph and $\pi:Y\to X$ be a free Galois covering of hypergraphs. Let $G$ be the Galois group of $\pi$. 
\begin{enumerate}
    \item 
Let $\widehat{G}$ be a complete set of inequivalent irreducible representations of $G$. Then, one has the following factorization:
\[
\zeta_Y(u) = \prod_{\rho \in \widehat{G}} L(u,\rho,Y/X)^{d_\rho},
\]   
where $d_\rho$ is the dimension of $\rho$.
\item 
Let $\rho_G$ be the right regular representation of $G$. Then, one has
\begin{equation}
L(u,\rho_G,Y/X)=\zeta_Y(u).
\end{equation}
\end{enumerate}
\end{nothmb}
\bigskip

The paper is organized as follows. In Section \ref{section: prelim}, we review some backgrounds. In Section \ref{section: bipartite associated to hyp}, we investigate relations between hypergraph and associated bipartite graphs. In Section \ref{section: Artin-Ihara}, we introduce Artin-Ihara L-functions for hypergraphs. In Section \ref{section: properties}, we prove various properties of Artin-Ihara L-functions of hypergraphs. We also compute some examples.

\bigskip

\noindent \textbf{Acknowledgment} J.J. acknowledges the support of an AMS-Simons Research Enhancement Grant for Primarily Undergraduate Institution (PUI) Faculty during the writing of this paper. The authors extend their gratitude to Chris Eppolito and Jaehoon Kim for their valuable feedback on the initial draft. The authors also would like to thank the anonymous referees for their comments and suggestions in improving the manuscript.

\bigskip



\section{Preliminaries}\label{section: prelim}

In what follows all graphs are assumed to be undirected and finite. For a graph $X$, we let $E(X)$ be the set of edges and $V(X)$ be the set of vertices. Finally, we assume that all graphs are connected unless otherwise stated. 

\subsection{Ihara zeta functions for graphs}

In this subsection, we briefly review the definition of the Ihara zeta function of a graph.

\begin{mydef}\cite[Section 2]{terras2010zeta}\label{definition: graph prime cycles}
Let $X$ be a graph and $E=\{e_1,\dots,e_m\}$ be the set of edges of $X$. For an oriented edge $e=(u,v)$, let $i(e)=u$ and $t(e)=v$. We orient the edges of $X$ arbitrarily and define $e_j^{-1}$ to be the edge $e_j$ with the opposite orientation. 
\begin{enumerate}
\item 
A \emph{path} of $X$ is a sequence $P=(a_1, \dots, a_l)$ of oriented edges of $X$ (i.e. $a_i \in \{e_j^\pm\}_{{j=1,\dots,m}}$) such that $t(a_i)=i(a_{i+1})$ for $1\leq i \leq l-1$. The \emph{length} $\ell(P)$ is $l$. 
\item 
$C$ is a \emph{closed path} (or \emph{cycle}) if the starting vertex is the same as the terminal vertex. 
\item 
A path $P=(a_1, \dots, a_l)$  has a \emph{backtracking} (resp.~ \emph{tail}) if $a_{j+1}=a_j^{-1}$ for some $j=1,\dots,l-1$ (resp.~$a_l=a_1^{-1}$). 
\item 
A cycle $C$ is \emph{prime} if $C$ does not have a backtracking nor tail, and $C \neq D^n$ for any cycle $D$ and $n \in \mathbb{Z}_{>1}$. 	
\item 
For a prime cycle $C=(a_1, \dots, a_l)$, let $[C]$ be the equivalence class of $C$, where two closed cycles are equivalent if and only if we can get one from the other by changing the initial vertex. 
\end{enumerate}
\end{mydef}

\begin{mydef}\label{definition: Ihara zeta for graphs}
The \emph{Ihara zeta function} $\zeta_X$ for a graph $X$ is defined as follows:
\begin{equation}
\zeta_X(u):=\prod_{[C]} (1-u^{\ell(C)})^{-1}, 
\end{equation}
where the product runs through the equivalence classes of all prime cycles $C$ and $\ell(C)$ denotes the length of $C$. 
\end{mydef}

As in several other zeta functions, one may convert the product form into a summation form as follows:
\begin{equation}
\zeta_X(u)=\exp\left(\sum_{n \geq 1}\frac{a_n}{n}u^n\right), 
\end{equation}
where $a_n$ is the number of cycles of length $n$ without backtracking nor tails (see, \cite[pp 29]{terras2010zeta}). From the definition, one can observe that if a graph $G$ has a vertex $v$ of degree $1$, then the graph $G'$ obtained by removing $v$ along with the incidence edge will produce the same zeta function. Also, the following determinant formula is well-known. 

\begin{mythm}\cite{bass1992ihara} \label{theorem: three-term determinant}
Let $X$ be a graph. If $A$ (resp.~$D$) is the adjacency matrix (resp.~the degree matrix) of $X$, then one has the following
\begin{equation}
\zeta_X(u)^{-1}=(1-u^2)^{r-1}\det(I-Au+(D-I)u^2),
\end{equation} 
where $r=|E(X)|-|V(X)|+1$.
\end{mythm}

An intriguing fact about $\zeta_X(u)$ is that the class of graphs $X$ satisfying a version of the Riemann hypothesis is precisely Ramanujan graphs. Recall that for a $(q+1)$-regular graph $X$, the Ihara zeta function $\zeta_X(q^{-s})$ satisfies the Riemann hypothesis if and only if the following hold: for $s \in \mathbb{C}$,
\[
\textrm{if } \textrm{Re}(s) \in (0,1) \textrm{ and }  \zeta_X(q^{-s})=0, \textrm{ then }  \textrm{Re}(s)=1/2.
\]

Then, one has the following characterization of Ramanujan graphs. 

\begin{mythm}\cite{stark1996zeta}\label{theorem: ramanujan}
A $(q+1)$-regular graph $X$ is Ramanujan if and only if $\zeta_X(q^{-s})$ satisfies the Riemann hypothesis.
\end{mythm}

\subsection{Artin-Ihara L-functions for graphs} \label{subsection: L-functions for graphs} 
Let $G$ be a graph. For $v \in V(G)$, we let $N(v)$ be the subgraph of $G$ induced by the set of vertices adjacent to $v$. In other words, $N(v)$ is the neighborhood of $v$.

By a \emph{morphism} $\varphi:G_1 \to G_2$ of graphs $G_1$ and $G_2$, we mean a pair of functions $(\varphi_V,\varphi_E)$, where $\varphi_V:V(G_1) \to V(G_2)$ and $\varphi_E:E(G_1) \to E(G_2)$, such that if $v \in e$ in $G_1$, then $\varphi_V(v) \in \varphi_E(e)$ in $G_2$ for $\forall~v \in V(G_1)$ and $\forall~e \in E(G_1)$. In other words, $\varphi$ preserves incidence relations.

A covering morphism $\pi=(\pi_V,\pi_E):Y \to X$ of graphs $Y$ and $X$ is a morphism of graphs such that $\pi_V$ is a surjection and $\pi$ is a local isomorphism on neighborhoods. To be precise, the following is an isomorphism for all vertices $x \in V(X)$ and $y \in \pi_V^{-1}(x)$
\[
\pi\mid_{N(y)}:N(y) \to N(x).
\]
Note that in \cite{li2018hypergraph}, Li and Hou introduced a notion of hypergraph coverings by using the same idea (isomorphisms on neighborhoods) and studied their zeta functions. 

A \emph{free Galois covering} is a covering morphism $\pi:Y \to X$ satisfying some conditions, and it has an associated \emph{Galois group}. We refer the reader to \cite{zakharov2021zeta} or \cite{terras2007zeta} for the precise definitions and examples.\footnote{In \cite{terras2007zeta}, a free Galois covering is called an \emph{unramified Galois covering}. Since our construction was motivated by the construction in \cite{zakharov2021zeta}, we follow the terminology in \cite{zakharov2021zeta}.}

Now, let $\pi:Y \to X:=Y/G$ be a free Galois covering of graphs with the Galois group $G$. To define the Artin-Ihara L-function of $\pi$, one first has to define the \emph{Frobenius element} for each prime cycle of $X$. Here are the steps to find Frobenius elements. 
\begin{enumerate}
    \item 
Fix a spanning tree $T_X$ of $X$ and consider a connected lift $T$ (in $Y$) of $T_X$. 
\item 
For each $g \in G$, we define the \emph{sheet number} $g$ to be the tree $T_g=g(T)$ so that the identity sheet becomes $T_{\textrm{id}_G}=T$. 
\item 
The trees $\{T_g\}_{g \in G}$ form a spanning forest for $Y$. In particular, for each $v \in V(Y)$ there exists a unique $g \in G$ such that $v \in V({g(T)})$. We define the sheet number of $v$ to be $g$.
\end{enumerate}

Now, one can define the Artin-Ihara L-function for $\pi:Y \to Y/G$ as follows. 

\begin{mydef}\label{definition: graph L-function}
Let $\pi:Y\to X$ be a free Galois covering of graphs. Let $G$ be the Galois group of $\pi$ and let $\rho$ be a representation of $G$. The \emph{Artin-Ihara} $L$-function of $\pi$ is defined as follows:
\[
L(u,\rho,Y/X):=\prod_{[C]} \det (1-\rho(F(C,Y/X))u^{\ell(C)})^{-1},
\]
where the product runs for all equivalence classes of prime cycles $C$ of $X$ and for each equivalence class $[C]$ we pick an arbitrary representative $C$. Note that the Frobenius element $F(C,Y/X) \in G$ associated to a prime cycle $C$ is the sheet number of the terminal vertex of the unique lift of $C$ that starts on the identity sheet $T_{\textrm{id}_G}$.
\end{mydef}

As in the case for Dedekind zeta functions and Artin L-functions in Number theory, Ihara zeta functions and Artin-Ihara L-functions are closely related. For instance, one has the following interesting results, due to Stark and Terras \cite{stark2000zeta}, analogous to Dedekind zeta functions and Artin L-functions.

\begin{mythm}\cite{terras2010zeta}\label{theorem: L function theorems for graphs}
 Let $\pi:Y\to X$ be a free Galois covering of graphs. Let $G$ be the Galois group of $\pi$.
\begin{enumerate}
    \item 
Let $\rho_1$ and $\rho_2$ be representations of $G$. Then $L(u,\rho_1\oplus \rho_2,Y/X)=L(u,\rho_1,Y/X)L(u,\rho_2,Y/X)$.
\item 
$L(u,1_G,Y/X)=\zeta_X(u)$, where $1_G$ is the trivial representation.
\item 
$L(u,\rho_G,Y/X)=\zeta_Y(u)$, where $\rho_{G}$ is the right regular representation.
\item 
$\zeta_Y(u)=\prod_{\rho \in \widehat{G}} L(u,\rho,Y/X)^{d_\rho}$,
where $\widehat{G}$ is a complete set of inequivalent irreducible representations of $G$ and $d_\rho$ is the dimension of $\rho$.\footnote{This indeed follows from a more general result (the Induction property). See \cite[Proposition 18.10]{terras2010zeta}.} 
\end{enumerate}
 
\end{mythm}


\subsection{Ihara zeta functions for hypergraphs}

A graph can be considered as a pair $(V,E)$ of a nonempty finite set $V$ (vertices), and a set of \emph{unordered} pair $E \subseteq V \times V$ (edges).\footnote{$E$ is in fact a multiset when we consider multigraphs.} \emph{Hypergraphs} are generalizations of graphs allowing edges to be \emph{any} nonempty subset. 

\begin{mydef}
By a \emph{hypergraph} $H$ we mean a pair $(V,E)$ of nonempty finite sets $V$ (\emph{hypervertices}) and $E\subseteq 2^V$ (\emph{hyperedges}) such that $\bigcup\limits_{e \in E} e = V$.\footnote{We allow hyperedges to repeat, so strictly speaking $E$ is a multiset.} We say that a hypervertex $v$ is \emph{incident} to a hyperedge $e$ if $v \in e$. In the following, we will simply say vertices and edges instead of hypervertices and hyperedges. 
\end{mydef}

We recall the definition of bipartite graph, which will be used as an important tool to study hypergraphs.

\begin{mydef}
 A \emph{bipartite} graph is a graph $G$ in which $V(G)=V_1(G)\sqcup V_2(G)$ for some nonempty disjoint subsets $V_1(G),V_2(G) \subseteq V(G)$, and no two vertices within the same set $V_i(G)$ for $i=1,2$ are adjacent.
\end{mydef}

For a given hypergraph $H$, one can associate a bipartite graph $B_H$ which encodes incidence relations on $H$ as follows:
\begin{enumerate}
    \item 
$V(B_H)=V(H) \sqcup E(H)$,
\item 
As vertices of $B_H$, $v \in V(H)$ and $e \in E(H)$ are adjacent in $B_H$ if $v$ is incident to $e$ in $H$. There will be no edges between the vertices in $V(H)$ and between vertices in $E(H)$. 
\end{enumerate}

\begin{myeg}\label{example: hypergraph bipartite example}
Let $H$ be a hypergraph with the following vertices and edges:
\[
V(H)=\{v_1,\dots,v_6\}, \quad E(H)=\{e_1, e_2, e_3, e_4  \},
\]
where 
\[
e_1=\{v_1,v_2,v_3\},\quad e_2=\{v_2,v_3\},\quad  e_3=\{v_3,v_5,v_6\},\quad  e_4=\{v_4\}.
\]
Pictorially, $H$ is the following:
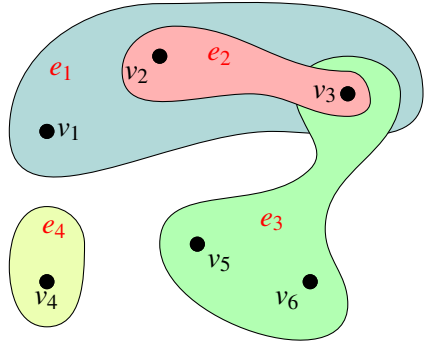
\[
\begin{tikzpicture}
    \node (v1) at (0,2) {};
    \node (v2) at (1.5,3) {};
    \node (v3) at (4,2.5) {};
    \node (v4) at (0,0) {};
    \node (v5) at (2,0.5) {};
    \node (v6) at (3.5,0) {};

    \begin{scope}[fill opacity=0.7]
    \filldraw[fill=teal!30] ($(v1)+(-0.5,0)$) 
        to[out=90,in=180] ($(v2) + (0,0.7)$) 
        to[out=0,in=90] ($(v3) + (1,0)$)
        to[out=270,in=0] ($(v2) + (1.5,-1)$)
        to[out=180,in=270] ($(v1)+(-0.5,0)$);
    \filldraw[fill=lime!30] ($(v4)+(-0.5,0.2)$)
        to[out=90,in=180] ($(v4)+(0,1)$)
        to[out=0,in=90] ($(v4)+(0.5,0.5)$)
        to[out=270,in=0] ($(v4)+(0,-0.6)$)
        to[out=180,in=270] ($(v4)+(-0.5,0.2)$);
    \filldraw[fill=green!30] ($(v5)+(-0.5,0)$)
        to[out=90,in=225] ($(v3)+(-0.5,-1)$)
        to[out=45,in=270] ($(v3)+(-0.7,0)$)
        to[out=90,in=180] ($(v3)+(0,0.5)$)
        to[out=0,in=90] ($(v3)+(0.7,0)$)
        to[out=270,in=90] ($(v3)+(-0.3,-1.8)$)
        to[out=270,in=90] ($(v6)+(0.5,-0.3)$)
        to[out=270,in=270] ($(v5)+(-0.5,0)$);
    \filldraw[fill=red!30] ($(v2)+(-0.5,-0.2)$) 
        to[out=90,in=180] ($(v2) + (0.2,0.4)$) 
        to[out=0,in=180] ($(v3) + (0,0.3)$)
        to[out=0,in=90] ($(v3) + (0.3,-0.1)$)
        to[out=270,in=0] ($(v3) + (0,-0.3)$)
        to[out=180,in=0] ($(v3) + (-1.3,0)$)
        to[out=180,in=270] ($(v2)+(-0.5,-0.2)$);
    \end{scope}
    \foreach \v in {1,2,...,6} {
        \fill (v\v) circle (0.1);
    }

    \fill (v1) circle (0.1) node [right] {$v_1$};
    \fill (v2) circle (0.1) node [below left] {$v_2$};
    \fill (v3) circle (0.1) node [left] {$v_3$};
    \fill (v4) circle (0.1) node [below] {$v_4$};
    \fill (v5) circle (0.1) node [below right] {$v_5$};
    \fill (v6) circle (0.1) node [below left] {$v_6$};

    \node at (0.2,2.8) {$\textcolor{red}{e_1}$};
    \node at (2.3,3) {$\textcolor{red}{e_2}$};
    \node at (3,0.8) {$\textcolor{red}{e_3}$};
    \node at (0.1,0.7) {$\textcolor{red}{e_4}$};
\end{tikzpicture}
\]
Then the associated bipartite graph $B_H$ is the following:
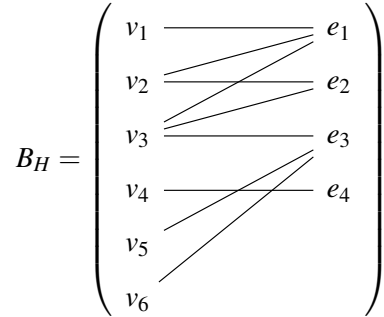
\[
B_H=\left(  
\begin{tikzcd}[every arrow/.append style={-}, row sep=0.2cm, column sep=1cm]
v_1 \arrow[rr] & & e_1 \\
v_2 \arrow[rr] \arrow[rru]& & e_2 \\
v_3  \arrow[rr] \arrow[rru] \arrow[rruu]& &e_3 \\
v_4 \arrow[rr] & &e_4 \\
v_5 \arrow[rruu] & & \\
v_6 \arrow[rruuu]& &
\end{tikzcd} 
\right)
\]
\end{myeg}

Here are some definitions and notations that we will use in the sequel. 

\begin{mydef}\cite{storm2006zeta}
Let $H$ be a hypergraph. 
\begin{enumerate}
    \item 
Let $u,v \in V(H)$. A \emph{path} of length $n$ from $u$ to $v$ is defined to be a sequence
\[
P=(u=v_0,e_1,v_1,e_2,v_2,\dots,e_n,v_n=v)
\]
such that $v_0 \in e_1$, $v_n \in e_n$, and $v_i \in e_{i+1} \cap e_i$ for all $i \in \{1,\dots,n-1\}$.\footnote{One may define this as a walk, but we follow the terminology in \cite{storm2006zeta}.}
\item 
We say that a path $P$ has \emph{edge-backtracking} if there is a subsequence of $P$ of the form $(e, v, e)$ for some $e \in E(H)$ and $v \in V(H)$.
\item 
If $n>1$ and $u=v$, then a path is called a \emph{cycle} or a \emph{closed path} of length $n$.
\item 
For a cycle $C$, by $C^\ell$ for $\ell \in \mathbb{N}$, we mean a cycle formed by going around the path $\ell$ times.
\item 
A cycle $C$ is said to be \emph{tail-less} if $C^2$ does not have edge-backtracking.
\item 
A cycle $C$ is \emph{prime} if $C$ does not have a backtracking nor a tail, and $C \neq D^n$ for any cycle $D$ and $n \in \mathbb{Z}_{> 1}$.
\end{enumerate}
\end{mydef}

\begin{mydef}
Let $H$ be a hypergraph. 
\begin{enumerate}
 \item 
$H$ is said to be without degree-1 vertices if any vertex $v \in V(H)$ is contained in at least two different edges.
\item 
$H$ is \emph{connected} if for all $u,v \in V(H)$ there is a path from $u$ to $v$. 
\end{enumerate}
\end{mydef}

\begin{mydef}\label{definition: equivalence class of cycles}
Let $H$ be a hypergraph and $C$ be a prime cycle. We let $C \sim Q$ if and only if they have the same underlying cycle, i.e.,
\[
\textrm{if} \quad C=(v_1,e_1,v_2,\dots,v_n,e_n,v_1), \quad \textrm{then} \quad Q=(v_i,e_i,\dots,v_n,e_n,v_1,e_1,\dots,e_{i-1},v_i)
\]
for some $i \in \{1,2,\dots,n\}$. We let $[C]$ be the equivalence class of $C$. 
\end{mydef}

\begin{mydef}\cite[Definition 4]{storm2006zeta} \label{definition: Storm zeta}
Let $H$ be a hypergraph. The \emph{Ihara zeta function} of $H$ is defined as follows
\begin{equation}
\zeta_H(u)=\prod_{[C]} (1-u^{\ell(C)})^{-1}, 
\end{equation}
where the product runs through all equivalence classes of prime cycles $C$ and $\ell(C)$ denotes the length of $C$. 
\end{mydef}

Storm showed the Ihara zeta function of $B_H$ is closely related to the Ihara zeta function of $H$ as follows. 

\begin{mythm}\cite[Theorem 10]{storm2006zeta} \label{theorem: zeta for hypergraoh via bipartite}
Let $H$ be a connected hypergraph, and $B_H$ be the associated bipartite graph. Then, one has
\begin{equation}
\zeta_H(u)=\zeta_{B_H}(\sqrt{u}),
\end{equation}
where $\zeta_{B_H}(u)$ is the Ihara zeta function of $B_H$.
\end{mythm}

\cite[Examples 15, 18]{storm2006zeta} shows that there exists a hypergraph whose zeta function is different from the Ihara zeta function of any graph. In particular, the zeta function of a hypergraph is a ``nontrivial'' generalization of the zeta function of a graph. Storm also showed an analog of Theorem \ref{theorem: ramanujan} for hypergraphs.

\section{Bipartite graphs associated to hypergraphs}\label{section: bipartite associated to hyp}

In this section, we study relations between hypergraphs and their associated bipartite graphs.

\begin{mydef}
Let $H_1$ and $H_2$ be hypergraphs. By a \emph{morphism} $\varphi:H_1 \to H_2$ we mean a pair of functions $(\varphi_V,\varphi_E)$, where $\varphi_V:V(H_1) \to V(H_2)$ and $\varphi_E:E(H_1) \to E(H_2)$, such that if $v \in e$ in $H_1$, then $\varphi_V(v) \in \varphi_E(e)$ in $H_2$ for $\forall~v \in V(H_1)$ and $\forall~e \in E(H_1)$. 
\end{mydef}

\begin{mydef}\label{definition: induced map}
Let $\varphi:H_1 \to H_2$ be a morphism of hypergraphs. Then, $\varphi$ induces a pair $\varphi_*=(f_V,f_E)$ of functions $f_V:V(B_{H_1}) \to V(B_{H_2})$ and $f_E:E(B_{H_1}) \to E(B_{H_2})$ such that:
\begin{equation}\label{eq: def of fv}
f_V(a) = \begin{cases}
    \varphi_V(a) & \textrm{ if $a \in V(H_1)$,}\\
    \varphi_E(a) & \textrm{ if $a \in E(H_1)$.}
\end{cases}
\end{equation}
For $f_E$, if there exists an edge $\alpha \in E(B_{H_1})$ between $a \in V(H_1)$ and $e \in E(H_1)$ then $\varphi_V(a) \in \varphi_E(e)$. Hence there exists a unique edge $\tilde{\alpha}$ in $B_{H_2}$ whose vertices are $\varphi_V(a)$ and $\varphi_E(e)$. We define $f_E(\alpha)=\tilde{\alpha}$. 
\end{mydef}

\begin{lem}\label{lemma: iso to iso}
With the same notation as in Definition \ref{definition: induced map}, if $\varphi$ is an isomorphism, then $\varphi_*=(f_V,f_E)$ is an isomorphism. 
\end{lem}
\begin{proof}
Suppose that $\varphi$ is an isomorphism, i.e., $\varphi_V$ and $\varphi_E$ are bijections. It is clear that $f_V$ is a bijection. Now, suppose that $f_E(\alpha)=f_E(\beta)$, where $\alpha = (a,e) \in V(H_1)\times E(H_1)$ and $\beta=(b,h) \in V(H_2)\times E(H_2)$. From the definition of $f_E$, we have that $\varphi_V(a)=\varphi_V(b)$ and $\varphi_E(e)=\varphi_E(h)$. Since $\varphi_V$ and $\varphi_E$ are bijections, it follows that $(a,e)=(b,h)$, showing that $f_E$ is an injection. On the other hand, suppose that $\gamma \in E(B_{H_2})$. It means that there exist $c \in V(H_2)$ and $k \in E(H_2)$ such that $c \in k$ and $\gamma=(c,k)$. Since $\varphi$ is an isomorphism, there exists $x \in V(H_1)$ and $\ell \in E(H_1)$ such that $\varphi_V(x)=c$ and $\varphi_E(\ell)=k$. It follows that there exists $\delta \in E(B_{H_1})$ such that $f_E(\delta)=\gamma$, showing that $f_E$ is a surjection as well. 
\end{proof}

\begin{pro}\label{proposition: faithful functor}
Let $\mathcal{H}$ be the category of hypergraphs and $\mathcal{G}$ be the category of graphs. The functor $\mathbf{B}:\mathcal{H} \to \mathcal{G}$, sending any hypergraph $H$ to $B_H$ and a morphism $\varphi:H_1 \to H_2$ to $\varphi_*:B_{H_1} \to B_{H_2}$, is faithful. 
\end{pro}
\begin{proof}
We first prove that $\mathbf{B}$ is indeed a functor. Let $\varphi:H_1\to H_2$ and $\psi:H_2 \to H_3$ be morphisms of hypergraphs. We claim that 
\[
(\psi\circ\varphi)_* = \psi_* \circ \varphi_*.
\]
Let $\Phi=(\psi\circ\varphi)$ and $\Phi_*=(h_V,h_E)$. We further let $\varphi_*=(f_V,f_E)$ and $\psi_*=(g_V,g_E)$. Firstly, it is clear from \eqref{eq: def of fv} that $h_V=g_V\circ f_V$. To show that $h_E=g_E \circ f_E$, suppose that $\alpha \in E(B_{H_1})$, i.e., $\alpha=(a,e) \in V(H_1) \times E(H_1)$ such that $a \in e$. Now, we have
\[
\Phi_V(a)=\psi_V(\varphi_V(a)) \in \Phi_E(e) = \psi_E(\varphi_E(e)),
\]
showing that $h_E(\alpha)=g_E(f_E(\alpha))$. Hence $\mathbf{B}$ is a functor.

Now, let $\varphi_1$ and $\varphi_2$ be distinct morphisms from $H_1$ to $H_2$. Then, we have either $(\varphi_1)_V\neq (\varphi_2)_V$ or $(\varphi_1)_E\neq (\varphi_2)_E$. This means that there exists $v\in V(H_1)$ so that $(\varphi_1)_V(v)\neq (\varphi_2)_V(v)$ or there exists $e\in E(H_1)$ so that $(\varphi_1)_E(e)\neq (\varphi_2)_E(e)$. Either way, there exists $u\in V(B_{H_1})$ so that $((\varphi_1)_*)_V(u)\neq ((\varphi_2)_*)_V(u)$. Hence, different maps from $H_1$ to  $H_2$ induce different maps from $B_{H_1}$ to $B_{H_2}$, showing that the functor $\mathbf{B}$ is faithful.
\end{proof}

The following simple example shows that $\mathbf{B}$ does not have to be full, even when we restrict ourselves to graphs in $\mathcal{H}$. 

\begin{myeg}
Consider the following hypergraphs $Y=(V_Y,E_Y)$ and $X=(V_X,E_X)$:
\[
V_Y=\{v_1,v_2\}, \quad E_Y=\{e_1\}, \quad V_X=\{u_1\}, \quad E_X=\{f_1\},
\]
where
\[
e_1=\{v_1,v_2\}, \quad f_1=\{u_1\}.
\]
Then, we have the following:
\[
B_Y=\left(\begin{tikzcd}[every arrow/.append style={-}, row sep=0.2cm, column sep=1cm]
    v_1 \arrow[r] & e_1 \\
    v_2  \arrow[ru]
\end{tikzcd}\right), \quad B_X=\left( \begin{tikzcd}[every arrow/.append style={-}, row sep=0.2cm, column sep=1cm]
    v_1 \arrow[r] & e_1 
\end{tikzcd} \right)
\]
Once can easily see that there are two graph morphisms from $B_Y$ to $B_X$, however there exists a unique morphism from $Y$ to $X$. 
\end{myeg}

Now, we define group actions on hypergraphs and quotient hypergraphs.

\begin{mydef}\label{definition: free action}
Let $G$ be a group and $H$ be a hypergraph.
\begin{enumerate}
    \item 
By an \emph{action} of $G$ on $H$, we mean a group homomorphism $f:G \to \Aut(H)$. 
\item 
An action $f:G \to \Aut(H)$ is said to be \emph{free} if the following two conditions hold:
\begin{enumerate}
    \item 
$G$ acts freely on $V(H)$ and $E(H)$.
\item 
For any $v \in V(H)$, if $v \in e$, then $g(v) \not \in e$ for all $g\neq \textrm{id}_G$. 
\end{enumerate}
\end{enumerate}
\end{mydef}

\begin{rmk}
Note that the notion of free action in Definition \ref{definition: free action} is identical to the notion of free action in \cite{zakharov2021zeta} when $H$ is a graph. 
\end{rmk}

Let $G$ be a group and $H$ be a hypergraph. A group action $f:G \to \Aut(H)$ induces a group action $f_*:G \to \Aut(B_H)$. In fact, it follows from Lemma \ref{lemma: iso to iso} that we have an injective homomorphism of groups $B_*:\Aut(H) \to \Aut(B_H)$. By composing this with $f$, we obtain a group action $B_*\circ f:G \to \Aut(B_H)$. The following shows that the same holds for free group actions. 

\begin{pro}\label{proposition: free induced action}
Let $G$ be a group acting freely on a hypergraph $H$. Then, the induced $G$-action on $B_H$ is also free. 
\end{pro}
\begin{proof}
Suppose that $G$ does not act freely on $E(B_H)$. There exist $g\in G$ and $\alpha\in E(B_H)$ such that $g\neq \textrm{id}_G$ and $g(\alpha)=\alpha$. Let $\alpha=(v,e)$, where $v \in V(H)$ and $e\in E(H)$. Since $G$ acts freely on $V(H)$ and $E(H)$, and also $g \neq \textrm{id}_G$, we have that $g(v) \neq v$ and $g(e) \neq e$. It follows that $g(\alpha) \neq \alpha$ by Definition \ref{definition: induced map}. This gives a contradiction, and hence $G$ acts freely on $E(B_H)$. For $V(B_H)=V(H) \sqcup E(H)$, since $G$ acts freely on $V(H)$ and $E(H)$, clearly $G$ acts freely on $V(B_H)$. This shows that $G$ acts freely on $B_H$.

It remains to check the condition (b). Let $V(B_H)=V_1 \sqcup V_2$, where $V_1=V(H)$ and $V_2=E(H)$. Notice that for any $v \in V_i$ and $g \in G$, we have that $g(v) \in V_i$ for $i=1,2$. Since no two vertices within the same set $V_i$ are adjacent, there is no $\alpha \in E(B_H)$ such that $v,g(v) \in \alpha$. This shows that the induced $G$-action on $B_H$ satisfies (b).
\end{proof}

Let $G$ be a group acting on a hypergraph $H$. Let $V(H/G):=V(H)/G$, the set of orbits of $V(H)$ under the action of $G$. Likewise, we let $E(H/G):=E(H)/G$, the set of orbits of $E(H)$ under the action of $G$. We let $[e]$ be the orbit of $e \in E(H)$ in $E(H/G)$. We also let $[v]$ be the orbit of $v \in V(H)$ in $V(H/G)$. For $e \in E(H)$, we define the following set:
\begin{equation}\label{eq: quotient edges}
V_e:=\{[v] \mid v \in e\}
\end{equation}

\begin{lem}\label{lemma: orbits}
With the same notation as above, the set $V_e$ only depends on $[e] \in E(H/G)$. 
\end{lem}
\begin{proof}
Suppose that $e'=ge$ for some $g \in G$. It is enough to show that $V_e \subseteq V_{e'}$. But, if $[v]\in V_e$, then there exists $h \in G$ such that $hv \in e$. Now, we have that $(gh)v \in ge=e'$, and $[(gh)v]=[v]$, showing that $V_e \subseteq V_{e'}$.
\end{proof}

From Lemma \ref{lemma: orbits}, we may consider $E(H/G)$ as a set of subsets of $V(E/H)$. Hence, we have the following definition of the quotient hypergraph $H/G$.

\begin{mydef}\label{definition: quotient hypergraph}
Let $G$ be a group acting on a hypergraph $H$. The \emph{quotient hypergraph} $H/G$ consists of vertices $V(H/G)$ and edges $E(H/G)$ (considered as a set of subsets of $V(E/H)$ as noted above). 
\end{mydef}

With Definition \ref{definition: quotient hypergraph}, there exists a natural morphism $\pi=(\pi_V,\pi_E):H \to H/G$ defined as follows: for vertices, any $v \in V(H)$ goes to the orbit $[v]$ and any $e \in E(H)$ goes to its orbit $[e]$ by using the identification in Lemma \ref{lemma: orbits}. It is clear that the map $\pi$ is well-defined. Moreover, one can easily see that it is a morphism of hypergraphs. In fact, suppose that $v \in e$ for $v \in V(H)$ and $e \in E(H)$. Now, we have to show that $\pi_V(v)=[v] \in \pi_E(e)=[e]$. But, this is clear from Lemma \ref{lemma: orbits}.

\begin{myeg}\label{example: quotient used}
Consider the hypergraph $Y=(V_Y,E_Y)$ with
\[
V_Y=\{v_1,v_2,\dots,v_6\}, \quad E_Y=\{e_1,e_2,e_3,e_4\},
\]
where
\[
e_1=\{v_2,v_3,v_5\}, \quad e_2=\{v_1,v_4,v_6\}, \quad e_3=\{v_2,v_4\}, \quad e_4=\{v_1,v_3\}.
\]
Pictorially, $Y$ is as follows:
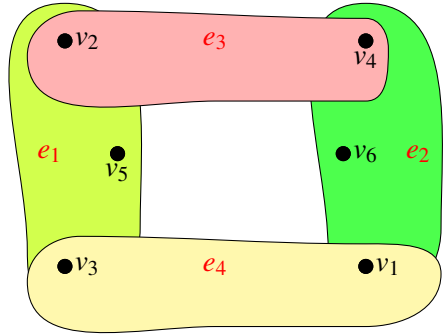
\[
\begin{tikzpicture}
    \node (v1) at (4,0) {};
    \node (v2) at (0,3) {};
    \node (v4) at (4,3) {};
    \node (v3) at (0,0) {};
    \node (v5) at (0.7,1.5) {};
    \node (v6) at (3.7,1.5) {};

 \begin{scope}[fill opacity=0.3]   
 \filldraw[fill=lime!70] ($(v3)+(-0.5,0.2)$)
        to[out=90,in=180] ($(v3)+(0,3.5)$)
        to[out=0,in=90] ($(v3)+(1,0.5)$)
        to[out=270,in=0] ($(v3)+(-0.1,-0.6)$)
        to[out=180,in=270] ($(v3)+(-0.5,0.2)$);
 \filldraw[fill=green!70] ($(v1)+(-0.5,0.2)$)
        to[out=90,in=180] ($(v1)+(0,3.5)$)
        to[out=0,in=90] ($(v1)+(1,0.5)$)
        to[out=270,in=0] ($(v1)+(-0.1,-0.6)$)
        to[out=180,in=270] ($(v1)+(-0.5,0.2)$);
  \filldraw[fill=red!30] ($(v2)+(-0.5,-0.2)$) 
        to[out=90,in=180] ($(v2) + (0.2,0.4)$) 
        to[out=0,in=180] ($(v4) + (0,0.3)$)
        to[out=0,in=90] ($(v4) + (0.3,-0.1)$)
        to[out=270,in=0] ($(v4) + (0,-0.8)$)
        to[out=180,in=0] ($(v4) + (-2,-0.8)$)
        to[out=180,in=270] ($(v2)+(-0.5,-0.2)$); 
  \filldraw[fill=yellow!40] ($(v3)+(-0.5,-0.2)$) 
        to[out=90,in=180] ($(v3) + (0.2,0.4)$) 
        to[out=0,in=180] ($(v1) + (0,0.3)$)
        to[out=0,in=90] ($(v1) + (1,-0.1)$)
        to[out=270,in=0] ($(v1) + (0,-0.8)$)
        to[out=180,in=0] ($(v1) + (-2,-0.8)$)
        to[out=180,in=270] ($(v3)+(-0.5,-0.2)$); 
    \end{scope}
 \foreach \v in {1,2,...,6} {
        \fill (v\v) circle (0.1);
    }
   \fill (v1) circle (0.1) node [right] {$v_1$};
    \fill (v2) circle (0.1) node [right] {$v_2$};
    \fill (v3) circle (0.1) node [right] {$v_3$};
    \fill (v4) circle (0.1) node [below] {$v_4$};
    \fill (v5) circle (0.1) node [below] {$v_5$};
    \fill (v6) circle (0.1) node [right] {$v_6$};

 \node at (-0.2,1.5) {$\textcolor{red}{e_1}$};
    \node at (4.7,1.5) {$\textcolor{red}{e_2}$};
    \node at (2,3) {$\textcolor{red}{e_3}$};
    \node at (2,0) {$\textcolor{red}{e_4}$};
\end{tikzpicture}
\]
Consider the automorphism $f=(f_V,f_E)$ defined as follows:
\[
f_V(v_1)=v_2, \quad f_V(v_2)=v_1, \quad f_V(v_3)=v_4, \quad f_V(v_4)=v_3, \quad f_V(v_5)=v_6, \quad f_V(v_6)=v_5,
\]
\[
f_E(e_1)=e_2, \quad f_E(e_2)=e_1, \quad f_E(e_3)=e_4, \quad f_E(e_4)=e_3.
\]
Let $G=\{\textrm{id}_G,f\}$. One can check that $G$ acts freely on $Y$. Hence, one obtains $X=Y/G=(V_X,E_X)$:
\[
V_X=\{u_1,u_2,u_3\}, \quad E_X=\{f_1,f_2\}, \quad \textrm{where}\quad f_1=\{u_1,u_2\}, \quad f_2=\{u_1,u_2,u_3\}.
\]
Note that as the orbits of $Y$ under $G$-action, we have
\[
u_1=\{v_1,v_2\}, \quad u_2=\{v_3,v_4\}, \quad u_3=\{v_5,v_6\}, \quad f_1=\{e_3,e_4\}, \quad f_2=\{e_1,e_2\}.
\]
Pictorially, $X=Y/G$ is as follows:
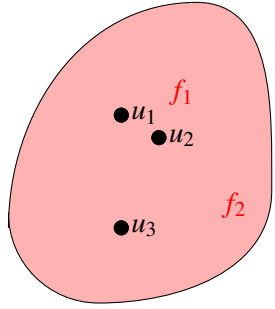
\[
\begin{tikzpicture}
    \node (u1) at (0,1.2) {};
    \node (u2) at (0.5,0.9) {};
    \node (u3) at (0,-0.3) {};

 \begin{scope}[fill opacity=0.3]   
 \filldraw[fill=green!70]
        ($(u1)+(-0.5,0.2)$)
        to[out=90,in=180] ($(u1)+(0.3,0.5)$)
        to[out=0,in=90] ($(u1)+(1.4,0.5)$)
        to[out=270,in=0] ($(u1)+(-0.5,-1)$)
        to[out=180,in=270] ($(u1)+(-0.5,0.2)$);

 \filldraw[fill=red!30]  
     ($(u1)+(-1.5,-1.5)$)
        to[out=90,in=180] ($(u3)+(1,3)$)
        to[out=0,in=90] ($(u3)+(2,0.5)$)
        to[out=270,in=0] ($(u3)+(-0.3,-1)$)
        to[out=180,in=270] ($(u3)+(-1.5,0.2)$);
    \end{scope}
 \foreach \v in {1,2,3} {
        \fill (u\v) circle (0.1);
    }
   \fill (u1) circle (0.1) node [right] {$u_1$};
    \fill (u2) circle (0.1) node [right] {$u_2$};
    \fill (u3) circle (0.1) node [right] {$u_3$};

 \node at (0.8,1.5) {$\textcolor{red}{f_1}$};
    \node at (1.5,0) {$\textcolor{red}{f_2}$};
\end{tikzpicture}
\]
\end{myeg}

\begin{lem}
Let $G$ be a group acting freely on a hypergraph $H$ without degree-1 vertices. Then, the quotient $H/G$ is also a hypergraph without degree-1 vertices.
\end{lem}
\begin{proof}
Suppose that $H/G$ has a degree-1 vertex, i.e., there exists a vertex $[v]\in V(H/G)$ which only belongs to one edge $[e] \in E(H/G)$. Now, it is clear that $v \in g(e)$ for some $g \in G$. Hence we may assume that $v \in e$. Since $H$ does not have a degree-1 vertex, there exists another edge $e'$ such that $v \in e'$. We claim that $[e'] \neq [e]$. In fact, if $[e']=[e]$, then $h(e)=e'$ for some $h \neq \textrm{id}_G\in G$. It follows that $h(v) \in h(e)=e'$ since $v \in e$. In particular, we have $v, h(v) \in e'$, contradicting the condition (b) in Definition \ref{definition: free action}.
\end{proof}

\begin{lem}\label{lemma: gal lift}
Let $G$ be a group acting freely on a hypergraph $H$, and $\pi:H \to H/G$ be the projection map. Let $v_0 \in V(H/G)$. Fix a vertex $v_0' \in V(H)$ such that $\pi(v_0')=v_0$. Then, for any $v_1 \in V(H/G)$ and an edge $e \in Y/G$ containing $v_0,v_1$, there exists a unique edge $e' \in (Y)$ such that $v_0' \in e'$ and $\pi(e')=e$. 
\end{lem}
\begin{proof}
First, choose $e'\in E(H)$ so that $\pi(e')=e$ and $v'_0\in e'$. Because $v_0\in e$, there exists such $e'$. To show uniqueness of such $e'$, suppose that $e''\in E(H)$ is another edge that satisfies the same properties as $e'$, i.e., 
\[
e'' \neq e', \quad \pi(e'')=\pi(e')=e, \quad v_0' \in e'', \quad v_0' \in e'.
\]
Since $\pi(e'')=\pi(e')$, we know that $e''$ and $e'$ belong to the same $G$-orbit, i.e., there exists $g\in G$ such that 
\[
g\neq \textrm{id}_G, \quad \textrm{and} \quad g(e'')=e'.
\]
Since $v'_0\in e''$, we have that 
\[
g(v'_0)\in g(e'')=e'.
\]
Hence both $v'_0$ and $g(v'_0)$ are incident to $e'$. This contradicts the condition (b) in Definition \ref{definition: free action}, showing that the lift $e'$ of $e$ is unique.    
\end{proof}

\begin{pro}\label{proposition: galois is cover}
If $G$ acts freely on a hypergraph $H$, then $\pi:H \to H/G$ is a covering in the sense of \cite{li2018hypergraph}.
\end{pro}
\begin{proof}
By definition of $\pi$, $\pi_V$ is surjective. Hence we only have to prove that the following
\[
\pi|_{N(v_0')}:N(v_0') \to N(v_0)
\]
is an isomorphism for all $v_0 \in V(H/G)$ and $v_0' \in \pi^{-1}_V(v_0)$. 

We first show that $\pi|_{N(v_0')}$ is a bijection on vertices. If $[u] \in N(v_0)=N([v_0'])$, then there exists $[e] \in E_{H/G}$ such that $[v_0'],[u] \in [e]$. It follows that $v_0' \in ge$ for some $g \in G$. Now, $ge$ contains $hu$ for some $h \in G$: if not $\pi(ge)=[e] \not \ni [u]$, giving us a contradiction. So, we have $v_0',hu \in ge$, and hence $hu \in N(v_0')$, showing that $\pi|_{N(v_0')}$ is surjective on vertices.  

To see that $\pi|_{N(v_0')}$ is injective on vertices, suppose that there exist $u,u' \in N(v_0')$ such that 
\begin{equation}\label{eq: rmk}
\pi(u)=\pi(u')=v_1 \in N(v_0).
\end{equation}
Then, we have $e \in E(H/G)$ containing both $v_0$ and $v_1$. It follows from Lemma \ref{lemma: gal lift} that there exists a unique edge $e'$ such that $v_0' \in e'$ and $\pi(e')=e$. We claim that $u,u' \in e'$. In fact, since $u,u' \in N(v_0')$, there exist $e_1,e_2 \in E(H)$ such that
\[
u,v_0' \in e_1, \quad u',v_0' \in e_2. 
\]
It follows that $\pi(e_1)$ and $\pi(e_2)$ contains $v_0$ and $v_1$, i.e., $e_1$ and $e_2$ are the unique lift $e'$ of $e$ in Lemma \ref{lemma: gal lift}, showing that $u,u' \in e'$. Moreover, from \eqref{eq: rmk}, we have $u=gu'$ for some $g \in G$. But, since $G$ acts freely on $H$, by Definition \ref{definition: free action} (2)(b), we have that $g=\textrm{id}_G$, or $u=u'$.

Next, we prove that $\pi|_{N(v_0')}$ is a bijection on edges. We first show that it is injective. For distinct $e',e'' \in E(H)$, suppose that 
\[
\pi(e')=\pi(e'')=e,
\]
where $v_1, v_0 \in e$ for some $v_1 \in N(v_0)$ and $v_0' \in e', e''$. Then, there exists $g \neq \textrm{id}_G \in G$ such that $e''=ge'$. It follows that $v_0' \in e'$ and $v_0' \in ge'$. Equivalently,
\[
g^{-1}(v_0') \in e' \quad \textrm{and} \quad v_0' \in e'.
\]
This contradicts Definition \ref{definition: free action} (2)(b). Hence $\pi|_{N(v_0')}$ is injective on edges.

It remains to show that $\pi|_{N(v_0')}$ is surjective on edges. But, this directly follows from the proof of Lemma \ref{lemma: gal lift}.
\end{proof}

\section{Artin-Ihara L-functions for hypergraphs}\label{section: Artin-Ihara}

In this section, we introduce Artin-Ihara L-functions for hypergraphs. There are several (inequivalent) definitions of a spanning tree of a hypergraph. Moreover, it is well-known that a connected hypergraph may not have a spanning tree, depending on a definition of spanning trees. For example, suppose that we define a spanning tree of a hypergraph $H$ to be a sub-hypergraph $T$ satisfying the following two conditions: (1) $T$ contains all vertices of $H$, and (2) for any $v_1,v_2 \in V(T)$, there exists a unique path from $v_1$ to $v_2$. Now, consider the hypergraph $H$ on four vertices with all possible edges with precisely three vertices. To find a spanning tree of $H$ (with the above definition), one has to pick at least two edges but any two edges form a cycle. To avoid this issue, we use instead the associated bipartite graphs and their spanning trees to define Frobenius elements in the setting of hypergraphs.

We first prove that for a hypergraph covering $\pi:Y \to X$ we can lift a path $P$ in $X$ to a path $P'$ in $Y$ in a unique way once we fix a lift of the initial vertex of $P$. It then follows from Proposition \ref{proposition: galois is cover} that the same property (uniqueness of lifts of paths) holds for $\pi:H \to H/G$ when $G$ acts freely on $H$.

\begin{lem}\label{lemma: unique path lift}
Let $\pi:Y \to X$ be a hypergraph covering as in \cite{li2018hypergraph}. Let $P=(v_0,e_1,v_1,e_2,v_2,...,e_n,v_n)$ be a path in $X$ starting at a vertex $v_0$. Fix a vertex $v'_0$ of $H$ such that $\pi(v'_0)=v_0$. There is a unique path $P'$ in $Y$ starting at $v_0'$ such that $\pi(P')=P$. 
\end{lem}
\begin{proof}
The proof is straightforward. To be precise, since $v_0,v_1 \in e_1$ and we fixed $\pi(v_0')=v_0$, there exist unique $v_1' \in V(Y)$ and $e_1' \in E(Y)$ such that $\pi(v_1')=v_1$ and $\pi(e_1')=e_1$ since a hypergraph covering is an isomorphism when we restrict it to neighborhoods. We can repeat the process of lifting edges and vertices, until we get $P'$.
\end{proof}

\begin{lem}\label{lemma: sheets lemma}
Let $H$ be a connected hypergraph. Then, the following hold.
\begin{enumerate}
\item 
If a group $G$ acts freely on $H$, then $H/G$ is connected. 
    \item 
$B_H$ is connected. 
\item 
Let $T$ be a spanning tree of $B_H$ and $v_0 \in V(H)$. For each $v_i\in V(H)$, there is a unique path from $v_0$ to $v_i$, such that the corresponding path in $B_{H}$ is contained in $T$.
\end{enumerate}
  
\end{lem}
\begin{proof}
$(1)$: This is clear as we may find a path in $H$, and consider $\pi(H)$ in $H/G$.

$(2)$: Let $u,v \in V(B_H)$. If both $u,v$ are from $V(H)$, then this is clear as $H$ is connected. If both $u,v$ are from $E(H)$, then we choose vertices $v_1 \in u$ and $v_2 \in v$ in $H$. Now, we take a path from $v_1$ to $v_2$ in $H$, then we add two more edges from $v_1$ to $u$ and $v_2$ to $v$ in $B_H$. This produces a path from $u$ to $v$ in $B_H$. The last case is when $u$ is from $V(H)$ and $v$ is from $E(H)$. In this case, we may also choose a vertex $v'\in v$ in $H$ so that we can reduce this to the first case. 

$(3)$: Since $T$ is a spanning tree, we have $V(T)=V(B_H)$. Let $v_i \in V(T)=V(B_H)$. Since $T$ is a tree, there exists a unique path $P$ from $v_0$ to $v_i$, say
\[
P=(v_0,\alpha_0,e_0,\beta_0,v_1,\alpha_1,e_1,\beta_1,v_2,\alpha_2,e_2,\beta_2,\dots,v_{i-1},\alpha_{i-1},e_{i-1},\beta_{i-1},v_i),
\]
where $\alpha_k=(v_k,e_k),\beta_k=(e_k,v_{k+1})\in E(T)$, 
for some $e_k \in E(H)$ and $v_k, v_{k+1} \in e_k$ for $k=0,\dots, i-1$.  It follow that 
\[
(v_0,e_0,v_1,\dots,e_{i-1},v_i)
\]
is a path from $v_0$ to $v_i$ in $H$. Now, the uniqueness is clear since if there are two different paths from $v_0$ to $v_i$ in $H$, then it gives us two different paths in $T$ from $v_0$ to $v_i$ (considered as vertices of $T$). 
\end{proof}

\begin{rmk}
Lemma \ref{lemma: sheets lemma} (3) suggests a potential definition of spanning trees of hypergraphs. To be precise, let $H$ be a hypergraph. Fix a spanning tree $T$ in $B_H$. Then, from Lemma \ref{lemma: sheets lemma} (3), there is a distinguished way to construct a path between any two vertices of $H$ so that the corresponding path in $B_H$ lines in $T$. Let $P$ be the union of these paths. Then $P$ is a sub-hypergraph of $H$ with the following properties:
\begin{enumerate}
   \item 
$P$ contains all vertices of $H$. 
\item 
Any two vertices of $H$ are connected by a unique path in $P$. 
\end{enumerate}
Lemma \ref{lemma: sheets lemma} (3) shows the existence of such $P$. For instance, consider the following hypergraph $H$, which we wrote at the beginning of Section 4:
\[
V(H)=\{v_1,v_2,v_3,v_4\}, \quad E(H)=\{e_1,e_2,e_3,e_4\}, \quad e_i = V(H) - \{v_i\}.
\]
Now consider the following spanning tree $T$ of $B_H$:
\[
V(T)=\{v_1,v_2,v_3,v_4,e_1,e_2,e_3,e_4\}
\]
\[
E(T)=\{(v_1,e_2),(e_2,v_3),(v_3,e_4),(e_4,v_2,),(v_2,e_1),(e_1,v_4),(v_4,e_3)\}.
\]
One can check that P defined above is the following:
\[
V(P)=\{v_1,v_2,v_3,v_4\}, \quad E(P)=\{e_1',e_2',e_3',e_4'\},
\]
where
\[
e_1'=\{v_2,v_4\}, \quad e_2'=\{v_1,v_3\}, \quad e_3'=\{v_4\}, \quad e_4'=\{v_2,v_3\}.
\]
We are grateful to an anonymous referee for pointing this out. 
\end{rmk}

We can now partition the vertices of a hypergraph into sheets.\footnote{We use the same terminology as in \cite{terras2010zeta} and \cite{zakharov2021zeta}.} As before, let $G$ be a group acting freely on a connected hypergraph $H$. We cannot lift a spanning tree of $H/G$ since not all hypergraphs have spanning trees. We use instead a spanning tree of the associated bipartite graph $B_{H/G}$. Let $\pi:H \to H/G$ be a natural projection map. Here are the steps. 

\begin{construction}[Sheet numbers]\label{construction: sheet partition} $ $
\begin{enumerate}
    \item 
Choose a spanning tree $T$ of $B_{H/G}$. This is possible by Lemma \ref{lemma: sheets lemma}.
\item 
Choose arbitrary $v_0\in V(H/G)$ and $v'_0\in V(H)$ so that $\pi(v'_0)=v_0$.
\item 
For each $v_i\in V(H/G)$, there is a unique path from $v_0$ to $v_i$, such that the corresponding path in $B_{H/G}$ is contained in $T$. Call this path $P_i$. This is possible by Lemma \ref{lemma: sheets lemma} applied to $H/G$. 
\item 
Lift $P_i$ to a unique path $P'_i$ in $H$ with initial vertex $v'_0$. This is possible by Lemma \ref{lemma: unique path lift}. Do this for all $v_i\in V(H/G)$.
\item 
The set of the terminal vertices of all $P'_i$ will be the \emph{sheet} corresponding to $\textrm{id}_G \in G$. The \emph{sheet number} of some other $v\in V(H)$ will be $g\in G$ such that there exists $v'_i$ in sheet $\textrm{id}_G$ so that $g(v'_i)=v$.
\end{enumerate}
\end{construction}

\begin{lem}
With the same notation as above, the sheet numbers partition the vertices $V(H)$. 
\end{lem}
\begin{proof}
We have to prove that each vertex $v \in V(H)$ has a unique sheet number. For the uniqueness, suppose that $g(v_i')=v=h(v_j')$ for some $g,h \in G$. Then, we have
\[
v_i=\pi(g_i(v_i')) = \pi(v)=\pi(g_j(v_j')) = v_j,
\]
showing that $v_i'=v_j'$ since, by definition, they are the terminal vertices of the unique lifts of paths from $v_0$ to $v_i$ and $v_0$ to $v_j$. It follows that $h^{-1}g=\textrm{id}_G$ as $G$ acts freely on $V(H)$, showing the uniqueness.

For existence, we have that $[v] \in V(T)$, where $[v]$ is the $G$-orbit of $v$. Then, from Construction \ref{construction: sheet partition}, we have $\pi(v)=v_i=\pi(v_i')$ for some $i$. It follows that $v$ and $v_i'$ belong to the same $G$-orbit, i.e., there exists $g \in G$ such that $g(v_i')=v$. 
\end{proof}

\begin{myeg}\label{example: sheet partition}
Consider $H$ and $X=H/G$ in Example \ref{example: quotient used}. We have the following bipartite graph associated to $X$.
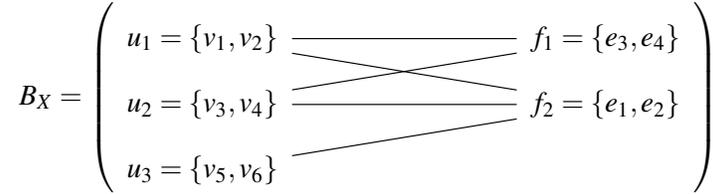
\[
B_X=\left(  
\begin{tikzcd}[every arrow/.append style={-}, row sep=0.2cm, column sep=1cm]
u_1=\{v_1,v_2\} \arrow[rrr] \arrow[rrrd]& & & f_1=\{e_3,e_4\} \\
u_2=\{v_3,v_4\} \arrow[rrru] \arrow[rrr] & & &   f_2=\{e_1,e_2\}\\
u_3=\{v_5,v_6\}  \arrow[rrru]  &  & &
\end{tikzcd} 
\right)
\]
Let $v_1 \in V_H$ be a fixed lift of $u_1$. For the spanning tree $T$, use the following spanning tree of $B_X$:
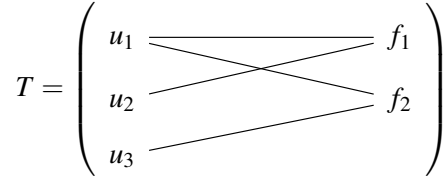
\[
T=\left(  
\begin{tikzcd}[every arrow/.append style={-}, row sep=0.2cm, column sep=1cm]
u_1 \arrow[rrr] \arrow[rrrd]& & & f_1 \\
u_2 \arrow[rrru]  & & &   f_2\\
u_3 \arrow[rrru]  &  & &
\end{tikzcd} 
\right)
\]
Clearly, $v_1$ has sheet number $\textrm{id}_G$. Now, the path from $u_1=\{v_1,v_2\}$ to $u_2=\{v_3,v_4\}$ contained in $T$ goes through $f_1=\{e_3,e_4\}$. The lift of this path starting at $v_1$ will end at $v_3$, so $v_3$ also has sheet number $\textrm{id}_G$. The path from $u_1=\{v_1,v_2\}$ to $u_3=\{v_5,v_6\}$ contained in $T$ goes through $f_2=\{e_1,e_2\}$. The lift of this path starting at $v_1$ ends at $v_6$, so $v_6$ has sheet number $\textrm{id}_G$. Because $G$ has order $2$, all of the other vertices of $H$ have sheet number $g$. See Figure \ref{figure1}. 

\begin{figure}[hbt!]
\[
\begin{tikzpicture}
    \node (v1) at (4,0) {};
    \node (v2) at (0,3) {};
    \node (v4) at (4,3) {};
    \node (v3) at (0,0) {};
    \node (v5) at (0.7,1.5) {};
    \node (v6) at (3.7,1.5) {};

 \begin{scope}[fill opacity=0.3]   
 \filldraw[fill=lime!70] ($(v3)+(-0.5,0.2)$)
        to[out=90,in=180] ($(v3)+(0,3.5)$)
        to[out=0,in=90] ($(v3)+(1,0.5)$)
        to[out=270,in=0] ($(v3)+(-0.1,-0.6)$)
        to[out=180,in=270] ($(v3)+(-0.5,0.2)$);
 \filldraw[fill=green!70] ($(v1)+(-0.5,0.2)$)
        to[out=90,in=180] ($(v1)+(0,3.5)$)
        to[out=0,in=90] ($(v1)+(1,0.5)$)
        to[out=270,in=0] ($(v1)+(-0.1,-0.6)$)
        to[out=180,in=270] ($(v1)+(-0.5,0.2)$);
  \filldraw[fill=red!30] ($(v2)+(-0.5,-0.2)$) 
        to[out=90,in=180] ($(v2) + (0.2,0.4)$) 
        to[out=0,in=180] ($(v4) + (0,0.3)$)
        to[out=0,in=90] ($(v4) + (0.3,-0.1)$)
        to[out=270,in=0] ($(v4) + (0,-0.8)$)
        to[out=180,in=0] ($(v4) + (-2,-0.8)$)
        to[out=180,in=270] ($(v2)+(-0.5,-0.2)$); 
  \filldraw[fill=yellow!40] ($(v3)+(-0.5,-0.2)$) 
        to[out=90,in=180] ($(v3) + (0.2,0.4)$) 
        to[out=0,in=180] ($(v1) + (0,0.3)$)
        to[out=0,in=90] ($(v1) + (1,-0.1)$)
        to[out=270,in=0] ($(v1) + (0,-0.8)$)
        to[out=180,in=0] ($(v1) + (-2,-0.8)$)
        to[out=180,in=270] ($(v3)+(-0.5,-0.2)$); 
    \end{scope}

   \fill (v1) circle (0.1) node [right] {$v_1$};
    \fill (v3) circle (0.1) node [right] {$v_3$};
    \fill (v6) circle (0.1) node [right] {$v_6$};

\node at (0.3,3) {$\textcolor{blue}{\bullet}~v_2$};
\node at (3.9,3) {$\textcolor{blue}{\bullet}~v_4$};
\node at (0.6,1.2) {$\textcolor{blue}{\bullet}~v_5$};

 \node at (-0.2,1.5) {$\textcolor{red}{e_1}$};
    \node at (4.7,1.5) {$\textcolor{red}{e_2}$};
    \node at (2,3) {$\textcolor{red}{e_3}$};
    \node at (2,0) {$\textcolor{red}{e_4}$};
\end{tikzpicture}
\]
\caption{Sheet $g$ is in blue.}  \label{figure1}
\end{figure}
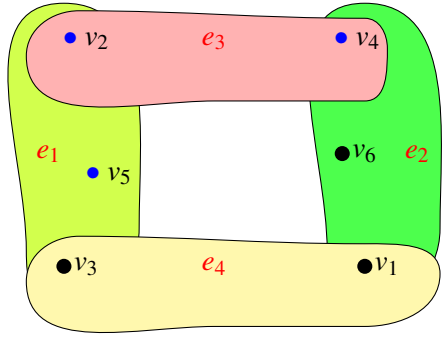
\end{myeg}

In the following we will denote a hypergraph by $Y$ and the quotient by $X=Y/G$.

\begin{mydef}\label{definiion: Frobenius}
Let $G$ be a group acting freely on a connected hypergraph $Y$. Fix $v_0$, $v_0'$, and $T$ as in Construction \ref{construction: sheet partition}. Let $C$ be a closed path of $X$. Let $P$ be the unique lifting of $C$ so that the starting vertex of $P$ has sheet number $\textrm{id}_G$. The \emph{Frobenius element} $F(C,Y/X)$ is the sheet number of the terminal vertex of $P$. 
\end{mydef}

\begin{rmk}\label{remark: Frobenius paths}
We remark the following observations. 
\begin{enumerate}
    \item 
Once one fixes $v_0$, $v_0'$, and $T$, then one can generalize the definition of a Frobenius element to all paths using the above definition by taking the sheet number of the terminal vertex of a unique lift.
\item 
For a closed path $C=(v_1,e_1,\dots,e_n,v_1)$ we can always find the unique lift $P'$ so that the starting vertex of $P'$ has sheet number $\textrm{id}_G$. To be precise, from (3) of Construction \ref{construction: sheet partition}, there is a unique path $P$ from $v_0$ to $v_1$ in $B_{H/G}$ which is contained in $T$. Then, by $(4)$ and $(5)$ of Construction \ref{construction: sheet partition}, we can find the unique lift $P'$ of $P$ whose initial vertex is $v_0'$. Moreover, by the definition of sheet $\textrm{id}_G$, the initial vertex of $P'$, which is a lifting of the vertex $v_1$ has a sheet number $\textrm{id}_G$.
\end{enumerate}
\end{rmk}

The following is a key result to define Frobenius elements in the setting of hypergraphs.

\begin{pro}\label{proposition: path sheet}
Let $G$ be a group acting freely on a connected hypergraph $Y$ and $X=Y/G$. Fix a spanning tree $T$ of $B_X$ and $v_0 \in V(X)$, $v_0'\in V(Y)$ as in Construction \ref{construction: sheet partition}. Two vertices $v_1,v_2\in V(Y)$ have the same sheet number if and only if there exists a path $P$ between them with $\pi(P)$ in $T$ (viewed in $B_{X}$).
\end{pro}
\begin{proof}
($\implies$) Suppose that $v_1$ and $v_2$ have sheet number $g$, i.e., there exist vertices $\tilde{v}_1$ and $\tilde{v}_2$ of $Y$ with sheet number $\textrm{id}_G$ so that $g(\tilde{v}_1)=v_1$ and $g(\tilde{v}_2)=v_2$. Since $\tilde{v}_1$ has the sheet number $\textrm{id}_G$, there exists $v_1^*\in V(X)$ with a unique path $P_1$ from $v_0$ to $v_1^*$ in $X$ such that the corresponding path in $B_{X}$ is contained in $T$, and there is a unique lift $P_1'$ of $P_1$ to $Y$ which has the terminal vertex  $\tilde{v}_1$. Likewise, for $\tilde{v}_2$, one has $v_2^*$, $P_2$, $P_2'$. Now, one can concatenate $P_1'$ and $P_2'$ to define a path between $v_1$ and $v_2$. To be precise, define a new path from $\tilde{v}_1$ to $\tilde{v}_2$:
\[
P':=(P_1')^{-1}P_2',
\]
which we read left-to-right in the order of traversal and $(P_1')^{-1}$ means that we walk ``backward''. To sum up, we have the following picture:
\begin{equation} \label{eq: path}
\begin{tikzcd}
   v_2 & \tilde{v}_2  \arrow[dashed,swap]{d}{\pi}\arrow[dashed,swap]{l}{g}  &   v_0'  \arrow[dashed,swap]{d}{\pi} \arrow[dashed]{r}{P_1'}  \arrow[dashed,swap]{l}{P_2'}    &\tilde{v}_1 	 \arrow[dashed,swap]{d}{\pi} 
 \arrow[bend right=45,dashed,swap]{ll}{P'}  \arrow[dashed]{r}{g}& v_1\\
    & v_2^*& v_0   \arrow[dashed,swap]{l}{P_2}   \arrow[dashed]{r}{P_1}& v_1^* &
\end{tikzcd}
\end{equation}
Then, $P=g(P')$ becomes a path from $v_1$ to $v_2$. Moreover, we can see that
\[
\pi(P)=\pi(g(P'))=\pi(P'),
\]
and hence $\pi(P)$ is in $T$ (viewed in $B_{X}$).

($\impliedby$) Let $v_1$ and $v_2$ be connected by a path $P$ so that $\pi(P)$ is in $T$ (viewed in $B_{X}$). Suppose that $v_1$ and $v_2$ have sheet numbers $g_1$ and $g_2$ with $g_1\neq g_2$. Let $\tilde{v}_2$ be $g_2g_1^{-1}(v_1)$. The sheet number of $\tilde{v}_2$ is $g_2$. Since $v_2$ and $\tilde{v}_2$ have the same sheet number, similar to the argument above as summarized in \eqref{eq: path}, there exists a path $P^*$, running from $v_2$ to $\tilde{v}_2$ with $\pi(P^*)\subseteq T$. Now, consider $P^{**}=PP^{*}$. Remove any backtrackings from $P^{**}$. Then, $\pi(P^{**})$ is a cycle of non-zero length, since $\pi(\tilde{v}_2)=\pi(g_2g_1^{-1}(v_1))=\pi(v_1)$, and $\tilde{v}_2\neq v_1$. The fact that $\pi(P^{**})\subseteq T$ is a contradiction.
\end{proof}

\begin{rmk}\label{remark: independent of lift}
 Let $G$ be a group acting freely on a connected hypergraph $Y$. Let $X=Y/G$. Fix a spanning tree $T$ of $B_{X}$. Let $P$ be a path in $X$, which is in $T$ (viewed in $B_{X}$). If $P'$ and $P''$ are two lifts of $P$ whose initial vertices have sheet number $\textrm{id}_G$, then the terminal vertices of $P'$ and $P''$ have the same sheet number. 
\end{rmk}

\begin{lem}\label{lemma: vertices choices}
With the same notation as above and a fixed spanning tree $T \subseteq B_{X}$, the Frobenius element $F(C,Y/X)$ depends on a choice of $v_0$ and $v_0'$ only up to conjugacy.
\end{lem}
\begin{proof}
We first fix $v_0$ and consider $v_0'$ and $v_0''$ such that $\pi(v_0')=\pi(v_0'')=v_0$. Since $v_0'$ and $v_0''$ are in the same $G$-orbit, $\exists~ h\in G$ so that $h(v_0')=v_0''$. Let $C=(v,e_1,v_2,\dots,e_n,v)$ be a closed path in $X=Y/G$, so that $F(C,Y/X)=g$ with respect to $v_0'$. This means if we lift $C$ to $C'$ in $Y$ so that the initial vertex of $C'$ has sheet number $\textrm{id}_G$ with respect to $v_0'$, then the terminal vertex of $C'$ will have sheet number $g$ with respect to $v_0'$.

Let $C'':=h(C')$ so that $\pi(C'')=C$. Let $i_{C''}$ (resp.~$t_{C''}$) be the initial (resp.~terminal) vertex of $C''$. Then, with respect to $v_0'$, the vertex $i_{C''}$ has sheet number $h$ and the vertex $t_{C''}$ has sheet number $hg$. Note that, with respect to $v_0'$, $v_0''$ has sheet number $h$. Since $v_0''$ and $i_{C''}$ have the same sheet number with respect to $v_0'$, and as the spanning tree $T$ is fixed, the vertex $i_{C''}$ has sheet number $\textrm{id}_G$ with respect to $v_0''$ by Proposition \ref{proposition: path sheet}. Therefore $C''$ is a unique lift of $C$ with respect to $v_0''$ as in Definition \ref{definiion: Frobenius}.

Since $t_{C''}$ has sheet number $hg$ with respect to $v_0'$, we have that $t_{C''}=hg(v_i')$ for some vertex $v_i'$ whose sheet number, with respect to $v_0'$, is $\textrm{id}_G$. Notice that $h(v_i')$ has sheet number $\textrm{id}_G$ with respect to $v_0''$, i.e., $h(v_i')=v_i''$ for some $v_i''$ whose sheet number, with respect to $v_0''$, is $\textrm{id}_G$. So, we have
\[
t_{C''} = hg(v_i')=hgh^{-1}(v_i''),
\]
showing that $t_{C''}$ has sheet number $hgh^{-1}$ with respect to $v_0''$. 

Next, we consider two choices of $v_0$, namely, $v_{0}^*$ and $v_{0}^{**}$. We can arbitrarily choose $(v_{0}^{*})', (v_{0}^{**})'\in V(Y)$ with respect to $v_{0}^*$ and $v_{0}^{**}$, so that $(v_{0}^{**})'$ has sheet number $\textrm{id}_G$ with respect to the choice of $v_{0}^*$ and $(v_{0}^{*})'$. It follows from Proposition \ref{proposition: path sheet} that there exists a path $P$ from $(v_{0}^{*})'$ to $(v_{0}^{**})'$, such that $\pi(P)$ is in $T$. Now, if there is a path $P^*$ from $v\in V(Y)$ to $(v_{0}^{*})'$ with $\pi(P^*)\subseteq T$, we can draw a similar path from $v$ to $(v_{0}^{**})'$ by concatenating paths $P$ and $P^*$, and removing backtrackings as we did in the proof of Proposition \ref{proposition: path sheet}. Hence, both $v_0^*$ and $v_0^{**}$ produce the same sheet $\textrm{id}_G$, and the same partition of the vertices of $Y$. Therefore, choice of $v_0$ does not change $F(C,Y/X)$.
\end{proof}

\begin{lem}\label{lemma: independent of tree}
With the same notation as above, the Frobenius element $F(C,Y/X)$ depends on a choice of a spanning tree $T \subseteq B_{Y/G}$ only up to conjugacy.
\end{lem}
\begin{proof}
Let $C=(v,e_1,\dots,e_n,v)$ be a closed path in $Y/G$. From Lemma \ref{lemma: vertices choices}, we know that choices of $v_0$ and $v_0'$ change $F(C,Y/X)$ only up to conjugacy. For convenience, let $v_0=v$, the initial (and terminal) vertex of $C$. Pick any $v_0'$ with $\pi(v_0')=v_0$. For any choice of a spanning tree $T$, $v_0'$ will have sheet number $\textrm{id}_G$. Furthermore, all $v\in V(Y)$ with $v=g(v_0')$ for some $g$ will have sheet number $g$. Now, we lift $C$ to $Y$ so that the initial vertex of $C$ goes to $v_0'$. Because $C$ is a cycle, the terminal vertex of $C$ will be lifted to $v^*$ so that $\pi(v^*)=v_0$. This means that $v^*=g(v_0')$ for some $g$. Regardless of our choice of $T$, $F(C,Y/X)=g$.
\end{proof}

Let $C \subseteq Y/G$ be a closed path, say $(v_1,e_1,v_2,e_2,\dots,e_n,v_1)$. We may change the starting vertex which produces the same underlying closed path. For instance, $(v_2,e_2,v_3,e_3,\dots,e_n,v_1,e_1,v_2)$ if we change our starting vertex to $v_2$. The following lemma says that this only changes the Frobenius element $F(C,Y/X)$ up to conjugacy. 

First recall that as noted in Remark \ref{remark: Frobenius paths}, one may consider the Frobenius element $F(P,Y/X)$ of a path $P$ in $Y/G$. 
The following lemma shows that taking Frobenius elements is multiplicative with respect to concatenation of paths. We let $i(P)$ be the initial vertex and $t(P)$ be the terminal vertex of a path $P$.

\begin{lem}\label{lemma: path homomorphism}
Consider paths $P_1,P_2$ of $Y/G$ such that the terminal vertex of $P_1$ is the initial vertex of $P_2$. Let $P:=P_1P_2$. Then, one has the following:
\begin{equation}
F(P,Y/X)=F(P_1,Y/X)F(P_2,Y/X).
\end{equation}
\end{lem}
\begin{proof}
Let $F(P_1,Y/X)=g_1$ and $F(P_2,Y/X)=g_2$. We first lift $P_1$ to $P_1'$ so that the initial vertex $i(P_1')$ has sheet number $\textrm{id}_G$. By definition, the vertex $t(P_1')$ has sheet number $g_1$. In other words, $t(P_1')=g_1(v)$ for some vertex $v$ with the sheet number $\textrm{id}_G$. Let $v^*=t(P_1)=i(P_2)$. Then, we have
\[
v^*=\pi (t(P_1'))=\pi (i(P_2')) = \pi (g_1(v)) = \pi(v).
\]
In particular, we can lift $P_2$ to $P_2'$ so that $i(P_2')=v$. Now, the terminal vertex of the path $g_1(P_2')$ has sheet number $g_1g_2$. We can conclude the proof by noticing that the path
$\tilde{P}:=P_1'g_1(P_2')$ is a lift of $P_1P_2$ whose initial vertex has sheet number $\textrm{id}_G$. 
\end{proof}

\begin{cor}\label{corollary: cycle reps class}
With the same notation as above, the Frobenius element $F(C,Y/X)$ depends on a choice of a representative of $C$ (with the equivalence relation in Definition \ref{definition: equivalence class of cycles}), but only up to conjugacy.   
\end{cor}
\begin{proof}
 Consider two representatives of a cycle $C$, which we will call $C_1$ and $C_2$. Let $P$ be a path from the initial vertex of $C_1$ to the initial vertex of $C_2$ that is also contained within $C$. It follows from Lemma \ref{lemma: path homomorphism} that 
 \[
F(C_1,Y/X)=F(P,Y/X)F(C_2,Y/X)F(P^{-1},Y/X)=F(P,Y/X)F(C_2,Y/X)F(P,Y/X)^{-1}. 
 \]
\end{proof}

\begin{mydef}\label{definition: Galois group}
Let $G$ be a group. By a \emph{free Galois covering} of hypergraphs with the Galois group $G$, we mean a morphism of hypergraphs $\varphi:Y \to X$ satisfying the following conditions:
\begin{enumerate}
    \item 
$G$ acts freely on $Y$,
    \item 
$X$ is isomorphic to $Y/G$, under which $\varphi$ is the projection map, and  
\item 
$G$ is isomorphic to $\{\sigma \in \Aut(Y) \mid \pi\circ \sigma = \pi\}$.
\end{enumerate}
\end{mydef}

\begin{rmk}\label{remark: inclusion remark}
In Definition \ref{definition: Galois group}, without the condition (3), one always has the following inclusion:
\[
G \hookrightarrow \{\sigma \in \Aut(Y) \mid \pi=\pi\circ \sigma\}.
\]
In fact, since $G$ acts on $Y$, we have a group homomorphism $f:G \to \Aut(Y)$. Since $G$ acts freely, $f$ is necessarily an injection. Since $\varphi=\pi:Y \to X=Y/G$, clearly $f(G) \subseteq \{\sigma \in \Aut(Y) \mid \pi=\pi\circ \sigma\}$. 
\end{rmk}

Now, we define the Artin-Ihara $L$-function of a free Galois covering of hypergraph. 

\begin{mydef}\label{definition: L-function}
Let $Y$ be a connected hypergraph and $\pi:Y\to X$ be a free Galois covering of hypergraphs. Let $G$ be the Galois group of $\pi$ and let $\rho$ be a representation of $G$. The \emph{Artin-Ihara} $L$-function of $\pi$ is defined as follows:
\[
L(u,\rho,Y/X):=\prod_{[C]} \det (1-\rho(F(C,Y/X))u^{\ell(C)})^{-1},
\]
where the product runs for all equivalence classes $[C]$ of prime cycles $C$ of $X$ (as in Definition \ref{definition: equivalence class of cycles}) and for each equivalence class $[C]$ we pick an arbitrary representative $C$. 
\end{mydef}

From Lemma \ref{lemma: vertices choices}, Lemma \ref{lemma: independent of tree}, and Corollary \ref{corollary: cycle reps class}, the Artin-Ihara $L$-function does not depend on any choices of $v_0,v_0',T$ in Construction \ref{construction: sheet partition} or a choice of a representative $C$ of $[C]$: they only change Frobenius elements up to conjugacy, and hence the determinant is well-defined. In other words, $L(u,\rho,Y/X)$ only depends on a representation $\rho$ and $\pi:Y  \to X$.

\section{Properties of Artin-Ihara L-functions of hypergraphs} \label{section: properties}

In this section, we explore properties of Artin-Ihara L-functions of hypergraphs. The following propositions are analogous to the graph case: See Proposition 18.10 in \cite{terras2010zeta} or the formulas at the end of Section 2.3 in \cite{zakharov2021zeta}.

\begin{pro}\label{proposition: L-function direct sum}
Let $Y$ be a connected hypergraph and $\pi:Y\to X$ be a free Galois covering of hypergraphs. Let $G$ be the Galois group of $\pi$ and $\rho_1$, $\rho_2$ be representations of $G$, then one has the following:
\[
L(u,\rho_1 \oplus \rho_2, Y/X) = L(u,\rho_1, Y/X) L(u, \rho_2, Y/X)
\]
\end{pro}
\begin{proof}
This directly follows from the definition of the direct sum of representations:
\[
(\rho_1\oplus\rho_2)(F(C,Y/X)) = \rho_1(F(C,Y/X)\rho_2(F(C,Y/X)).
\]
\end{proof}

\begin{pro}\label{proposition: trivial rep}
With the same notation as in Proposition \ref{proposition: L-function direct sum}, one has the following
\begin{equation}
L(u,1_G,Y/X)=\zeta_X(u),
\end{equation}
where $1_G$ is the trivial representation of $G$ and $\zeta_X(u)$ is the Ihara zeta function of $X$. 
\end{pro}
\begin{proof}
Recall that the trivial representations means that $1_G:G \to \textrm{GL}(k)=k^\times$ sending all $g$ to the identity element $1_k \in \textrm{GL}(k)$. So, with $1_G$, one has
\[
L(u,1_G,Y/X) = \prod_{[C]} \det (1-u^{\ell(C)})^{-1} = \prod_{[C]} (1-u^{\ell(C)})^{-1} =\zeta_X(u).
\]
The last equality holds since the product runs through all equivalence classes of prime cycles in $X$.
\end{proof}

In \cite{storm2006zeta}, Storm showed that the zeta function $\zeta_Y(u)$ of a connected hypergraph $Y$ and the zeta function $\zeta_{B_Y}(u)$ of the associated bipartite graph is related as follows:
\[
\zeta_Y(u)=\zeta_{B_Y}(\sqrt{u}).
\]
In the following, we prove that the same result holds for $L$-functions. Recall that if $\pi:Y \to Y/G$ is a free Galois covering of hypergraphs with the Galois group $G$, then the free $G$-action on $Y$ induces a free $G$-action on $B_Y$ by Proposition \ref{proposition: free induced action} and that $B_Y/G$ is the quotient by this induced action. We let $\tilde{\pi}:B_Y \to B_Y/G$ be the projection. We start with the following proposition.

\begin{pro}\label{proposition: Galois cover}
Let $G$ be a group. Let $Y$ be a connected hypergraph and $\pi:Y \to Y/G$ be a free Galois covering of hypergraphs with the Galois group $G$. Then, there exists a graph isomorphism $\phi:B_{Y/G} \to B_Y/G$ such that $\phi\circ\textbf{B}(\pi)=\tilde{\pi}$, where $\mathbf{B}$ is the functor in Proposition \ref{proposition: faithful functor}.
\end{pro}
\begin{proof}
We construct a graph isomorphism $\phi:B_{Y/G} \to B_Y/G$ such that $\phi\circ \tilde{\pi}=\textbf{B}(\pi)$. For any morphism $\sigma$ of hypergraphs, we let $\sigma_V$ (resp.~$\sigma_E$) be the map on vertices (resp.~edges). 

We first define $\phi_V:V(B_{Y/G}) \to V(B_Y/G)$. Let $v_1,v_2 \in V(B_Y)=V(Y)\sqcup E(Y)$. We claim that $\textbf{B}(\pi)(v_1)=\textbf{B}(\pi)(v_2)$ if and only if $\tilde{\pi}(v_1)=\tilde{\pi}(v_2)$. In fact, suppose that $\textbf{B}(\pi)(v_1)=\textbf{B}(\pi)(v_2)$. 
Since $\textbf{B}(\pi)$ is the morphism induced by $\pi$, either $v_1,v_2\in V(Y)$ and $\pi_V(v_1)=\pi_V(v_2)$ or $v_1,v_2\in E(Y)$ and $\pi_E(v_1)=\pi_E(v_2)$. Either way, $g(v_1)=v_2$ for some $g\in G$. It follows that $\tilde{\pi}(v_1)=\tilde{\pi}(v_2)$. On the other hand, suppose that $\tilde{\pi}(v_1)=\tilde{\pi}(v_2)$. Then, there exists $g\in G$ so that $g(v_1)=v_2$, and hence $\pi(v_1)=\pi(v_2)$. Since $\textbf{B}(\pi)$ is induced by $\pi$, this means that $\textbf{B}(\pi)(v_1)=\textbf{B}(\pi)(v_2)$. It now follows from the claim that $\textbf{B}(\pi)$ and $\tilde{\pi}$ partition $V(B_Y)$ into orbits in an identical way. Both $V(B_{Y/G})$ and $V(B_Y/G)$ are equal to the set of these orbits. Hence, there is an obvious one to one correspondence,
\[
\phi_V:V(B_{Y/G})\to V(B_Y/G),
\]
where $\phi_V(\textbf{B}(\pi)(v_0))=\tilde{\pi}(v_0)$ for any $v_0\in V(B_Y)$.

Next, we claim that vertices $v_1,v_2 \in V(B_{Y/G})$ are connected by an edge if and only if $\phi_V(v_1)$ and $\phi_V(v_2)$ are connected by an edge. Sine $B_{Y/G}$ and $B_Y/G$ are bipartite graphs, this will define a bijection $\phi_E:E(B_{Y/G}) \to E(B_Y/G)$, and $\phi=(\phi_V,\phi_E)$ will be the desired isomorphism. 

To prove the claim, first consider $v_1,v_2\in V(B_{Y/G})$ that are connected by an edge. We may assume that $v_1\in V(Y/G)$, $v_2\in E(Y/G)$ and $v_1$ is incident to $v_2$ in $Y/G$. Then, there exist $u_1\in \pi^{-1}(v_1)$ and $u_2\in \pi^{-1}(v_2)$ so that $u_1$ is incident to $u_2$. In particular, $\tilde{\pi}$ will map them to vertices connected by an edge in $B_Y/G$. Therefore, if $v_1$ and $v_2$ are connected by an edge, so are $\phi_V(v_1)$ and $\phi_V(v_2)$.

Conversely, consider $v_1,v_2\in V(B_Y/G)$ that are connected by an edge. There exist $u_1\in (\tilde{\pi})^{-1}(v_1)$ and $u_2\in (\tilde{\pi})^{-1}(v_2)$ that are connected by an edge. We may assume that $u_1\in V(Y)$, $u_2\in E(Y)$ and $u_1$ is incident to $u_2$. It follows that $\textbf{B}(\pi)(u_1)$ is incident to $\textbf{B}(\pi)(u_2)$, and hence $\phi_V^{-1}(v_1)$ and $\phi_V^{-1}(v_2)$ are connected by an edge as claimed. 
\end{proof}

\begin{lem}\label{lemma: v to v e to e}
Let $G$ be a group acting freely on a connected hypergraph $Y$, and let $\pi:Y \to Y/G$ and $\tilde{\pi}:B_Y \to B_Y/G$ be the corresponding quotient maps. If $\sigma \in \Aut(B_Y)$ such that $\tilde{\pi}\circ \sigma = \tilde{\pi}$, then $\sigma_V(V(Y)) = V(Y)$ and $\sigma_V(E(Y))=E(Y)$, where $V(Y)$ and $E(Y)$ are viewed as vertices of $B_Y$.
\end{lem}
\begin{proof}
Since an automorphism $\sigma$ preserves distances between vertices, $\sigma$ satisfies exactly one of the following two:
\begin{enumerate}
    \item 
$\sigma_V(V(Y))=V(Y)$ and $\sigma_V(E(Y))=E(Y)$, or
    \item 
$\sigma_V(V(Y))=E(Y)$ and $\sigma_V(E(Y))=V(Y)$.    
\end{enumerate}  
But, since $\tilde{\pi}_V(V(Y)) \subseteq V(Y)$ and $\tilde{\pi}_V(E(Y)) \subseteq E(Y)$ from the condition $\tilde{\pi}\circ \sigma = \tilde{\pi}$, we necessarily have the first case. 
\end{proof}

\begin{pro}\label{proposition: same group}
Let $G$ be a group acting freely on a connected hypergraph $Y$, and let $\pi:Y \to Y/G$ and $\tilde{\pi}:B_Y \to B_Y/G$ be the corresponding quotient maps. Then, the following is an isomorphism of groups:
\[
\Psi:I:=\{\delta \in \Aut(Y) \mid \pi\circ \delta= \pi\} \to J:=\{\sigma \in \Aut(B_Y) \mid \tilde{\pi}\circ \sigma = \tilde{\pi}\}, \quad \delta \mapsto \mathbf{B}(\delta),
\]
where $\mathbf{B}$ is the functor in Proposition \ref{proposition: faithful functor}.
\end{pro}
\begin{proof}
We first note that $\Psi$ is well-defined. To be precise, if $\delta \in \Aut(Y)$ such that $\pi\circ \delta = \pi$, then we have
\[
\mathbf{B}(\pi \circ \delta) =\mathbf{B}(\pi) \iff \mathbf{B}(\pi) \circ \mathbf{B}(\delta) =\mathbf{B}(\pi) \iff \tilde{\pi}\circ \mathbf{B}(\delta) = \tilde{\pi},
\]
where the first equivalence holds since $\mathbf{B}$ is a functor and the second equivalence follows from Proposition \ref{proposition: Galois cover}. Moreover, $\Psi$ is injective since the functor $\mathbf{B}$ is faithful (Proposition \ref{proposition: faithful functor}). 

Next, we prove that $\Psi$ is surjective. Take $\sigma=(\sigma_V,\sigma_E) \in J$. Let $V_1$ (resp.~$V_2$) be the set of the vertices of $B_Y$ obtained from the vertices (resp.~edges) of $Y$. It follows from Lemma \ref{lemma: v to v e to e} that $\sigma_{V_1}:V_1 \to V_1$ and $\sigma_{V_2}:V_2 \to V_2$ are bijections, where $\sigma_{V_i}=\sigma_V\mid_{V_i}$. Viewed $V_1$ and $V_2$ as vertices and edges in $Y$, we have $\sigma_{V_1}:V(Y) \to V(Y)$ and $\sigma_{V_2}:E(Y) \to E(Y)$. 

Let $\delta:=(\sigma_{V_1},\sigma_{V_2})$. We first claim that $\delta:Y \to Y$ is a morphism of hypergraphs. In fact, suppose that $x \in V(Y)$, $e \in E(Y)$ such that $x \in e$. This implies that there exists a unique edge $\alpha \in E(B_Y)$ connecting $x$ and $e$, viewed as vertices in $B_Y$. It follows that $\sigma_E(\alpha)$ is an edge of $B_Y$ connecting $\sigma_V(x)=\sigma_{V_1}(x)$ and $\sigma_V(e)=\sigma_{V_2}(e)$. In particular, viewed in $Y$, $\sigma_{V_1}(x) \in \sigma_{V_2}(e)$. In other words, $\delta$ preserves incidence relations on $Y$, and hence it is a morphism of hypergraphs. Moreover, since $\sigma_{V_1}$ and $\sigma_{V_2}$ are bijections, we have $\delta \in \textrm{Aut}(Y)$.

Next, we show that $\mathbf{B}(\delta)=\sigma$. In fact, from Definition \ref{definition: induced map}, if $\mathbf{B}(\delta)=(f_V,f_E)$, then clearly one has $f_V=\sigma_{V_1} \sqcup \sigma_{V_2} = \sigma_V$. To show $f_E=\sigma_E$, suppose that $\alpha \in E(B_Y)$ is obtained from an incidence relation between $x \in V(Y)$ and $e \in E(Y)$. Now, from the definition of $f_E$, the edge $f_E(\alpha)$ is obtained from an incidence relation between $\sigma_{V_1}(x) \in V(Y)$ and $\sigma_{V_2}(e) \in E(Y)$. Since $B_Y$ is a simple graph, this implies that $f_E(\alpha)=\sigma_E(\alpha)$. Hence, $\mathbf{B}(\delta)=\sigma$.

Finally, we have
\[
\tilde{\pi}\circ \sigma =\tilde{\pi} \iff \mathbf{B}(\pi)\circ \mathbf{B}(\delta) =\mathbf{B}(\pi\circ \delta) = \mathbf{B}(\mathbf{\pi}). 
\]
But, since the functor $\mathbf{B}$ is faithful, we have $\pi\circ \delta =\pi$, showing that $\delta \in I$ and hence $\Psi$ is surjective. 
\end{proof} 

\begin{cor}\label{corollary: same galois group}
Let $G$ be a group and $Y$ be a connected hypergraph. Let $\pi:Y \to Y/G$ be a free Galois covering of hypergraphs with the Galois group $G$. Then, the induced map $\textbf{B}(\pi):B_Y \to B_{Y/G}$ is also a free Galois covering of bipartite graphs with the Galois group $G$. 
\end{cor}
\begin{proof}
By Proposition \ref{proposition: Galois cover}, we can view $\textbf{B}(\pi)$ as $\tilde{\pi}:B_Y \to B_Y/G$, and by Proposition \ref{proposition: free induced action} the action of $G$ on $B_Y$ is free. Finally, from Proposition \ref{proposition: same group}, $G$ is isomorphic to $\{\sigma \in \Aut(B_Y) \mid \textbf{B}(\pi)\circ \sigma = \textbf{B}(\pi)\}$. 
\end{proof}

The following is a key result to link the Artin-Ihara $L$-function of a hypergraph $Y$ and that of the associated bipartite graph $B_Y$.

\begin{lem}\label{lemma: L-function key lemma}
Let $Y$ be a connected hypergraph and $\pi: Y\to Y/G$ be a free Galois covering of hypergraphs. There exist sheet partitions of $V(Y)$ and $V(B_Y)$ so that any $v_i\in V(Y)$ has the same sheet number whether it is viewed as a vertex in $V(Y)$ or viewed as a vertex in $B_Y$.
\end{lem}
\begin{proof}
We use Construction \ref{construction: sheet partition} to partition $V(Y)$, and we use $G$-action on $B_Y$ to partition $V(B_Y)$ (as in Section \ref{subsection: L-functions for graphs}).

For the sheet partition for $Y$, choose a spanning tree $T$ of $B_{Y/G}$ and vertices $v_0\in V(Y/G)$ and $v_0'\in V(Y)$ such that $\pi(v_0')=v_0$. Then we partition $Y$ into sheets following Construction \ref{construction: sheet partition}. 

For the sheet partition for $B_Y$, since $B_{Y/G}=B_Y/G$, we can use the same $T$ for the partitioning of $B_Y$. We can uniquely lift $T$ to $B_Y$ in such a way that the vertex $v_0$ in $V(B_Y/G)$ is lifted to the vertex $v_0'$ in $V(B_Y)$.

Consider $v_i\in V(Y)$ with sheet number $g$, i.e., there exists a vertex $v_i' \in V(Y)$ with sheet number $\textrm{id}_G$ so that $g(v_i')=v_i$. 
By Proposition \ref{proposition: path sheet}, there exists a path $P$ between $v_i'$ and the vertex $v_0'$ (viewed in $B_Y$) so that $\mathbf{B}(\pi)(P)$ is contained in $T$. This means that $v_i'$ has sheet number $\textrm{id}_G$ (viewed in $B_Y$). The way $g$ acts on $B_Y$ is induced by how $g$ acts on $Y$. In particular, $g(v_i')=v_i$ implies that $v_i$ has sheet number $g$ when it is viewed as a vertex in $B_Y$.
\end{proof}

\begin{lem} \label{lemma: cycle corr}
Let $Y$ be a connected hypergraph and $\pi: Y\to X=Y/G$ be a free Galois covering of hypergraphs. Let $C$ be a cycle in $Y/G$ and $B_C$ be the associated cycle of $C$ in $B_{Y/G}$. Then, one has the following:
\[
F(C,Y/X)=F(B_C,B_Y/B_X).
\]
\end{lem}
\begin{proof}
Partition $V(Y)$ and $V(B_Y)$ into sheets as in Lemma \ref{lemma: L-function key lemma}. Let $g=F(C,Y/G)$. Lift $C$ to $C'$ in $Y$ so that the initial vertex has sheet number $\textrm{id}_G$. The terminal vertex has sheet number $g$. Consider $B_{C'}$, the associated path to $C'$ in $B_Y$. By Proposition \ref{proposition: Galois cover}, $\mathbf{B}(\pi)(B_{C'})=C$. It follows from Lemma \ref{lemma: L-function key lemma} that the initial vertex of $B_{C'}$ has sheet number $\textrm{id}_G$ and the terminal vertex has sheet number $g$. Hence, $F(B_C,B_Y/B_X)=g$.
\end{proof}

\begin{mythm}\label{theorem: linking L-functions}
Let $Y$ be a connected hypergraph and $\pi:Y\to X$ be a free Galois covering of hypergraphs. Let $G$ be the Galois group of $\pi$ and $\rho$ be a representation of $G$. Then, one has the following:
\[
L(u,\rho,Y/X) = L(\sqrt{u},\rho,B_Y/B_X).
\]
\end{mythm}
\begin{proof}
From \cite[Proposition 9]{storm2006zeta}, we know that there is a one to one correspondence between prime cycles $C$ of length $\ell$ in $Y$ and prime cycles $B_C$ of length $2\ell$ in $B_Y$. Moreover, under this correspondence, we have $F(C,Y/X)=F(B_C,B_Y/B_X)$ by Lemma \ref{lemma: cycle corr}. Now, our assertion directly follows:
\[
L(u,\rho,Y/X)=\prod_{[C]} \det (1-\rho(F(C,Y/X))u^{\ell(C)})^{-1}=\prod_{[B_C]} \det (1-\rho(F(B_C,B_Y/B_X))u^{\ell(B_C/2)})^{-1}
\]
\[
= \prod_{[B_C]} \det (1-\rho(F(B_C,B_Y/B_X))(\sqrt{u})^{\ell(B_C)})^{-1}=L(\sqrt{u},\rho,B_Y/B_X).
\]
\end{proof}

\begin{rmk}
Storm showed in \cite[Examples 15 and 18]{storm2006zeta} that there are zeta functions of hypergraphs which do not arise from any graphs. Together with our results, this implies that there are $L$-functions of Galois coverings of hypergraphs which do not arise from any Galois covering of graphs. 
\end{rmk}

From Theorem \ref{theorem: linking L-functions}, we obtain the following results.

\begin{cor}\label{corollary: 1}(Factorization of the Ihara zeta function)
Let $Y$ be a connected hypergraph and $\pi:Y\to X$ be a free Galois covering of hypergraphs. Let $G$ be the Galois group of $\pi$. Let $\widehat{G}$ be a complete set of inequivalent irreducible representations of $G$. Then, one has the following factorization:
\[
\zeta_Y(u) = \prod_{\rho \in \widehat{G}} L(u,\rho,Y/X)^{d_\rho},
\]
where $d_\rho$ is the dimension of $\rho$. 
\end{cor}
\begin{proof}
By applying the same factorization for graphs (\cite[Corollary 18.11]{terras2010zeta}), we have
\[
\zeta_Y(u)=\zeta_{B_Y}(\sqrt{u})=\prod_{\rho \in \widehat{G}} L(\sqrt{u},\rho,B_Y/B_X)^{d_\rho} = \prod_{\rho \in \widehat{G}} L(u,\rho,Y/X)^{d_\rho}.
\]
\end{proof}

\begin{cor}\label{corollary: 2}
With the same notation and assumption as in Corollary \ref{corollary: 1}, one has the following
\begin{equation}
L(u,\rho_G,Y/X)=\zeta_Y(u),
\end{equation}
where $\rho_G$ is the right regular representation. 
\end{cor}
\begin{proof}
Notice that we have
\[
\zeta_Y(u)=\zeta_{B_Y}(\sqrt{u})=L(\sqrt{u},\rho_G,B_Y/B_X) =L(u,\rho_G,Y/X),
\]
where the first equality by from Storm \cite{storm2006zeta}, the second equality is from Theorem \ref{theorem: L function theorems for graphs} and Corollary \ref{corollary: same galois group}, and the last equality is from Theorem \ref{theorem: linking L-functions}.
\end{proof}

From Theorem \ref{theorem: linking L-functions}, we may use determinant formulas, such as  \cite[Theorem 18.8]{terras2010zeta}, for Artin-Ihara L-functions for graphs to compute those for hypergraphs.

Here is an example which we compute by using associated bipartite graphs. 

\begin{myeg}
Consider the following hypergraph $Y=(V_Y,E_Y)$:
\[
V_Y=\{v_1,v_2,v_3,v_4,v_1',v_2',v_3',v_4'\}, \quad E_Y=\{e_1,e_2,e_3,e_1',e_2',e_3'\},
\]
where
\[
e_1=\{v_1,v_2,v_3\},~e_2=\{v_1,v_2,v_4\},~e_3=\{v_3,v_4'\},~ e_1'=\{v_1',v_2',v_3'\},~ e_2'=\{v_1',v_2',v_4'\},~ e_3'=\{v_3',v_4\}.
\]
The associated bipartite graph $B_Y$ is the following:
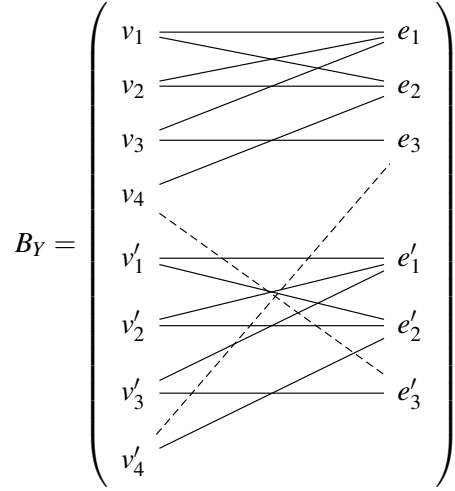
\[
B_Y=\left(  
\begin{tikzcd}[every arrow/.append style={-}, row sep=0.2cm, column sep=1cm]
v_1 \arrow[rrr]  \arrow[rrrd] & & & e_1 \\
v_2 \arrow[rrru]  \arrow[rrr]& & &e_2 \\
v_3  \arrow[rrruu]  \arrow[rrr] &  & &e_3 \\
v_4 \arrow[rrruu] \arrow[dashed, rrrddd]& & &  \\
v_1' \arrow[rrr]  \arrow[rrrd] & & & e_1'\\
v_2'   \arrow[rrru]  \arrow[rrr]& & & e_2' \\
v_3' \arrow[rrruu]  \arrow[rrr] & & & e_3'\\
v_4' \arrow[dashed, rrruuuuu] \arrow[rrruu]& & &
\end{tikzcd} 
\right)
\]
where two edges are dashed to avoid any confusion. 
One can easily see that $G=\{\textrm{id}_G,g\}\simeq \mathbb{Z}/2\mathbb{Z}$ acts freely on $Y$ (as in Definition \ref{definition: free action}) as follows:
\[
g(v_i)=v_i', \quad g(e_i)=e_i'.
\]
Hence, one obtains the quotient $X:=Y/G=(V_X,E_X)$:
\[
V_X=\{u_1,u_2,u_3,u_4\}, \quad E_X=\{f_1,f_2,f_3\} 
\]
where $u_i$ are the orbits of $v_i$ and $f_i$ are the orbits of $e_i$. To be precise, 
\[
f_1=\{u_1,u_2,u_3\}, \quad f_2=\{u_1,u_2,u_4\}, \quad f_3=\{u_3,u_4\}.
\]
We claim that $\pi:Y \to X$ is a Galois covering with the Galois group $G=\{\textrm{id}_G,g\}$. In fact, from Remark \ref{remark: inclusion remark}, there is a natural inclusion
\begin{equation}\label{eq: equality for groups}
G \hookrightarrow I=\{\sigma \in \Aut(Y) \mid \pi\circ \sigma = \pi\}.
\end{equation}
Hence, we only have to show the other inclusion. Suppose that $\varphi \in I$, i.e., we have two bijections
\[
\varphi_V:V_Y \to V_Y, \quad \varphi_E:E_Y \to E_Y,
\]
which satisfies incidence relations obtained from $Y$. To begin, one can check that there are two cases for $\varphi_V(v_1)$, namely $\varphi_V(v_1)=v_1$ or $\varphi_V(v_1)=v_1'$, since $\pi_V\circ \varphi_V = \pi_V$ and $\pi_V^{-1}(\pi_V(v_1))=\{v_1,v_1'\}$.\\

\noindent \underline{Case 1 ($\varphi_V(v_1)=v_1$):} We first observe that $\varphi_E(e_1)=e_1$. Indeed there are two cases, namely $\varphi_E(e_1)=e_1$ or $\varphi_E(e_1)=e_1'$. However, since $\varphi_V(v_1)=v_1$, we should have $v_1 \in \varphi_E(e_1)$, and hence $\varphi_E(e_1) = e_1$. It follows that
\[
\varphi(v_2)=v_2, \quad \varphi(v_3)=v_3.
\]
By applying the same argument to $v_3$ and $e_3$, we conclude that
\[
\varphi_E(e_3)=e_3, \quad \varphi_V(v_4')=v_4'.
\]
Doing this for $v_4'$ and $e_2'$ yields
\[
\varphi_E(e_2')=e_2', \quad \varphi_V(v_1')=v_1', \quad \varphi_V(v_2')=v_2'.
\]
By keeping doing this for $e_1', e_3'$ (and their vertices), we see that $\varphi$ is the identity of $\Aut(Y)$. \\

\noindent \underline{Case 2 ($\varphi_V(v_1)=v_1'$):} This is similar to the above case. In fact, in this case, we have $\varphi_E(e_1)=e_1'$ since $\varphi_V(v_1)=v_1'$. It follows that $\varphi_V(v_2)=v_2'$ and $\varphi_V(v_3)=v_3'$. Now, the same argument with $v_3$ and $e_3$ shows that $\varphi_E(e_3)=e_3'$, implying that $\varphi_V(v_4)=v_4'$. By keeping doing this, one can easily see that $\varphi$ is precisely the image of $g \in G=\{\textrm{id}_G,g\}$. Hence, we conclude that $G=I$ in \eqref{eq: equality for groups}, and hence $\pi$ is a Galois covering.

Now, the associated bipartite graph $B_X$ is the following:
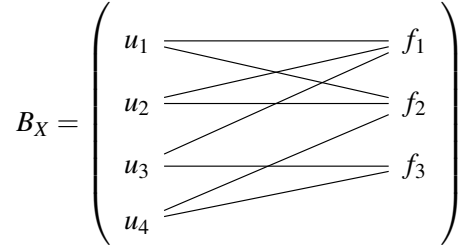
\[
B_X=\left(  
\begin{tikzcd}[every arrow/.append style={-}, row sep=0.2cm, column sep=1cm]
u_1 \arrow[rrr] \arrow[rrrd]& & & f_1 \\
u_2 \arrow[rrru]  \arrow[rrr]& & &   f_2 \\
u_3  \arrow[rrruu] \arrow[rrr]  &  & &f_3 \\
u_4  \arrow[rrruu]\arrow[rrru]  & & & 
\end{tikzcd} 
\right)
\]
One can easily see that $B_Y$ is a Galois covering of $B_X$ (two-sheeted covering). In particular, we can use the two-term determinant formula to compute the L-function. In the following, we use the same notation as in \cite{terras2010zeta}. 

We first need the edge adjacency matrix $W_1$ for $B_X$. Label the edges of $B_X$ as follows:
\[
\alpha_1=(u_1,f_1),~\alpha_2=(u_2,f_1),~\alpha_3=(u_3,f_1),~\alpha_4=(u_1,f_2)
\]
\[
\alpha_5=(u_2,f_2),~\alpha_6=(u_4,f_2),~\alpha_7=(u_3,f_3),~\alpha_8=(u_4,f_3).
\]
Then, we obtain the following $16\times 16$ matrix $W_1$:
\[
W_1=\left[\begin{array}{c|c}
\mathbf{0}_{8\times 8} & A_1 \\ \hline
B_1 & \mathbf{0}_{8\times 8}
\end{array}\right]
\]
where $\mathbf{0}_{8\times 8}$ is the zero matrix of size $8\times 8$ and 
\[
A_1=\begin{bmatrix}
0& 1& 1& 0& 0& 0& 0&0 \\
1& 0& 1& 0&0 &0 &0 &0 \\
1& 1& 0& 0& 0& 0&0 &0 \\
0&0 &0 &0 &1 &1 &0 &0 \\
0&0 &0 &1 &0 &1 &0 &0 \\
0& 0& 0& 1& 1& 0& 0& 0\\
0&0 &0 &0 &0 &0 &0 &1 \\
0& 0& 0& 0& 0& 0& 1& 0
\end{bmatrix}, \quad 
B_1=\begin{bmatrix}
0&0 &0 &1 &0 &0 & 0&0 \\
0&0 &0 &0 &1 &0 &0 &0 \\
0& 0& 0& 0& 0& 0&1 &0 \\
1& 0& 0& 0& 0& 0& 0&0 \\
0& 1& 0& 0& 0& 0& 0&0 \\
0&0 &0 &0 &0 & 0& 0& 1\\
0& 0& 1& 0& 0& 0& 0& 0\\
0& 0& 0& 0& 0& 1& 0&0 
\end{bmatrix}
\]

Let $G=\{\textrm{id}_G,g\}$. Then $G$ has two irreducible one-dimensional representations, namely the trivial representation and the sign representation.\\

\noindent \underline{Case 1 (Sign representation):}
Consider the following representation of $G=\{\textrm{id}_G,g\}$:
\[
\rho:G \to \textrm{GL}_1(\mathbb{R})=\mathbb{R}^\times, \quad g \mapsto -1.
\]
Now, we fix the sheet $S_{\textrm{id}_G}$ of $B_Y$ consisting of the vertices $\{v_1,v_2,v_3,v_4,e_1,e_2,e_3\}$. For each edge $\alpha_i$, the normalized Frobenius automorphism as in \cite[Definition 16.1]{terras2010zeta} is given as follows:
\[
\sigma(\alpha_1)=1,~ \sigma(\alpha_2)=1,~ \sigma(\alpha_3)=1,~ \sigma(\alpha_4)=1, 
\]
\[
\sigma(\alpha_5)=1,~ \sigma(\alpha_6)=1,~ \sigma(\alpha_7)=1,~ \sigma(\alpha_8)=-1, 
\]
Hence, we obtained the following matrix $W_{1,\rho}$:
\[
W_{1,\rho}=\left[\begin{array}{c|c}
\mathbf{0}_{8\times 8} & A_{1,\rho} \\ \hline
B_{1,\rho} & \mathbf{0}_{8\times 8}
\end{array}\right]
\]
where $\mathbf{0}_{8\times 8}$ is the zero matrix of size $8\times 8$ and 
\[
A_{1,\rho}=\begin{bmatrix}
0& 1& 1& 0& 0& 0& 0&0 \\
1& 0& 1& 0&0 &0 &0 &0 \\
1& 1& 0& 0& 0& 0&0 &0 \\
0&0 &0 &0 &1 &1 &0 &0 \\
0&0 &0 &1 &0 &1 &0 &0 \\
0& 0& 0& 1& 1& 0& 0& 0\\
0&0 &0 &0 &0 &0 &0 &1 \\
0& 0& 0& 0& 0& 0& -1& 0
\end{bmatrix}, \quad 
B_{1,\rho}=\begin{bmatrix}
0&0 &0 &1 &0 &0 & 0&0 \\
0&0 &0 &0 &1 &0 &0 &0 \\
0& 0& 0& 0& 0& 0&1 &0 \\
1& 0& 0& 0& 0& 0& 0&0 \\
0& 1& 0& 0& 0& 0& 0&0 \\
0&0 &0 &0 &0 & 0& 0& 1\\
0& 0& 1& 0& 0& 0& 0& 0\\
0& 0& 0& 0& 0& -1& 0&0 
\end{bmatrix}
\]
So, we have
\[
L(u,\rho,B_Y/B_X)^{-1}=\det(I-uW_{1,\rho})=1-2u^4+4u^6+u^8 -4u^{10}+4u^{12} - 4u^{16}.
\]
In particular, we have
\[
L(u,\rho,Y/X)=L(\sqrt{u},\rho,B_Y/B_X)=\frac{1}{1-2u^2+4u^3+u^4-4u^5+4u^6-4u^8}
\]
\[
=\frac{1}{(1-u)(1+u)^2(1-2u+2u^2)(1+u+2u^3)}.
\]
\bigskip

\noindent\underline{Case 2 (Trivial representation):} In this case, we have $\sigma(\alpha_i)=1$ for all $i$. Hence, we have $W_{1,1_G}=W_1$. So, we compute:
\[
L(u,1_G,B_Y/B_X)^{-1}=\det(I-uW_{1,1_G})=-4u^{16}+4u^{12}+4u^{10}+u^8-4u^6-2u^4+1.
\]
Hence, we have
\[
L(u,1_G,Y/X)=L(\sqrt{u},1_G,Y/X)=\frac{1}{(1-u)^2(1+u)(1+2u+2u^2)(1-u-2u^3)}.
\]

Next, to compute the zeta function $\zeta_X(u)$, we compute $\zeta_{B_X}(u)$ by using the determinant formula (see \cite[Theorem 2.5]{terras2010zeta}):
\[
\zeta_{B_X}(u)^{-1}=(1-u^2)^{r-1}\det (I-Au+Qu^2),
\]
where $r=|E|-|V|+1=2$, $A$ is the adjacency matrix of $B_X$ and $Q$ is the degree matrix of $B_X$. So, we have
\[
\zeta_{B_X}(u)=\frac{1}{(1-u^2)(4u^{14}+4u^{12}-4u^8-5u^6-u^4+u^2+1)},
\]
and hence
\[
\zeta_X(u)=\frac{1}{(1-u)(4u^{7}+4u^{6}-4u^4-5u^3-u^2+u+1)}=\frac{1}{(1-u)^2(1+u)(1-u-2u^3)(1+2u+2u^2)}.
\]
In particular, we see that $\zeta_x(u)=L(u,1_G,Y/X)$.

Finally, we compute $\zeta_Y(u)$ to confirm that $\zeta_Y(u)=L(u,1_G,Y/X)\cdot L(u,\rho,Y/X)$. By using the three-term determinant formula, we have
\[
\zeta_{B_Y}(u)^{-1}=
\]
\[
(1-u^2)^2(1+2u^2-u^4-4u^6-u^8+2u^{10}-7u^{12}-16u^{14}-16u^{16}-16u^{18}-8u^{20}+16u^{24}+32u^{26}+16u^{28}).
\]
Hence, we have
\[
\zeta_Y(u)^{-1}=
\]
\[
(1-u)^2(1+2u-u^2-4u^3-u^4+2u^5-7u^6-16u^7-16u^8-16u^9-8u^{10}+16u^{12}+32u^{13}+16u^{14})
\]
\[
=(1-u)^3(1+u)^3(1-2u+2u^2)(1+2u+2u^2)(1-u-2u^3)(1+u+2u^3).
\]
Hence, we have
\[
\zeta_Y(u)=\frac{1}{(1-u)^3(1+u)^3(1-2u+2u^2)(1+2u+2u^2)(1-u-2u^3)(1+u+2u^3)}.
\]
In particular, this shows that
\[
\zeta_Y(u)=L(u,1_G,Y/X)\cdot L (u,\rho,Y/X).
\]
\end{myeg}

The following is an example with hypergraphs with degree-1 vertices.

\begin{myeg}
Let $G$, $Y$, and $X$ be as in Example \ref{example: quotient used}. We denote $G$ by $\{\textrm{id}_G,g\}$. One may first observe that no prime cycles in $X=Y/G$ can contain $u_3$ since it would yield edge-backtracking. Hence, the only prime cycles in $X$ are represented by $C=(u_1,f_1,u_2,f_2,u_1)$ and $C'=(u_1,f_2,u_2,f_1,u_1)$. Using the sheet partitioning in Example \ref{example: sheet partition}, both $C$ and $C'$ have Frobenius element $g$. 

There are two inequivalent irreducible representations of $G$: the trivial representation and the sign representation $\rho$. For the trivial representation, we use Proposition \ref{proposition: trivial rep} to compute the $L$-function:
\[
L(u,\rho,Y/X)=\zeta_{X}(u)=\prod_{[P]} (1-u^{\nu(P)})^{-1}=(1-u^2)^{-2}.
\]
For the sign representation, we have
\[
L(u,\rho,Y/X)=\prod_{\mathfrak{c}} \det (1-\rho(F(C,Y/X))u^{\ell(\mathfrak{c})})^{-1}=\det (1-(-1)u^2)^{-2}=(1+u^2)^{-2}.
\]
Next, we consider prime cycles in $Y$. No prime cycle in $Y$ can contain $v_5$ or $v_6$ as that will yield edge-backtracking. Hence the only prime cycles in $Y$ are  represented by
\[
C=(v_1,e_4,v_3,e_1,v_2,e_3,v_4,e_2,v_1) \quad \textrm{and} \quad C'=(v_1,e_2,v_4,e_3,v_2,e_1,v_3,e_4,v_1).
\]
Therefore, we have
\[
\zeta_Y(u)=\prod_{[P]} (1-u^{\nu(P)})^{-1}=(1-u^4)^{-2}.
\]
Now, we check that 
\[
\zeta_Y(u)=\prod_{\rho \in \widehat{G}} L(u,\rho,Y/X)^{d_\rho}=((1-u^2)^{-2})^1((1+u^2)^{-2})^1=(1-u^4)^{-2}.
\]
\end{myeg}

\bigskip
\bigskip

\bibliography{zeta}\bibliographystyle{alpha}

\end{document}